\documentclass[10pt,final]{amsart} 
\usepackage{mystyle}

\usepackage{mathabx}
\newtheorem{conj}[thm]{Conjecture}

\newcommand{\PP}{\mathbb{P}}

\newcommand{\mc}[1]{\mathcal{#1}} % short for mathcal
\newcommand{\mt}[1]{\text{#1}}

\DeclareMathOperator{\Ker}{\mathsf{Ker}}

\DeclareMathOperator{\rank}{rank}
\DeclareMathOperator{\WW}{\mathsf{W}}
\DeclareMathOperator{\gr}{Gr}
\DeclareMathOperator{\Hom}{Hom}
\DeclareMathOperator{\stab}{Stab}

\DeclareMathOperator{\sd}{sd}

\DeclareMathOperator{\Sym}{\mathsf{Sym}}

\DeclareMathOperator{\chr}{char}

 \AtNextBibliography{\small}
\begin{document}

\title[Complete Quadrics]{Combinatorial models for the variety of complete quadrics}

\author{Soumya D. Banerjee}
\address{Ben-Gurion University of the Negev}
\email{soumya@math.bgu.ac.il}

\author{Mahir Bilen Can}
\author{Michael Joyce}
\address{Tulane University}
\email{mcan@tulane.edu and mjoyce3@tulane.edu}

\keywords{complete quadrics, Bia{\l}ynicki-Birula decomposition, Bruhat order, Richardson-Springer monoid}

\subjclass[2010]{19E08, 14M27}

\date{\today}

\begin{abstract}
We develop several combinatorial models that are used in the study of the variety of complete quadrics $\mcal{X}$. We
introduce the notion of a \emph{degenerate involution} and \emph{barred permutation} that parametrize geometrically
meaningful subsets of $\mcal{X}$. Using these combinatorial objects, we characterize particular
families of curves and surfaces on $\mcal{X}$ that are important for equivariant-cohomology calculations. We investigate the Bruhat order on Borel orbits in
$\mcal{X}$ and describe it in terms of (reverse) $\WW$-sets. Moreover, we prove (by a
counter example) that the Bruhat order induced from the symmetric group on $\mu$-involutions is
not isomorphic to the geometric Bruhat order on Borel orbits, unlike the case of ordinary involutions in
symmetric group. We also describe the Bia{\l}ynicki-Birula cell decomposition for $\mcal{X}$ in terms of the combinatorics
of degenerate involutions.  
\end{abstract}

\maketitle

%%% Local Variables:
%%% mode: latex
%%% TeX-master: t
%%% End:

%                 +++ Ending file MC-0.tex +++   
%------------------------------------------------------------------------------------------------------------------------------------------------

%                 +++ Starting file MC-0-intro.tex +++   

%  INTRO [[[

\section{Introduction}

The variety of complete quadrics $\mcal{X}$ has a venerable  place in classical algebraic geometry  alongside 
Grassmanians and flag varieties. It sits at the crossroads of algebraic geometry and representation theory appearing on one hand as a
parameter space in classical enumerative problems \cite{Chasles64} and on the other hand as an early motivating example
of the wonderful compactification of a symmetric space \cite{DP83}. However, our knowledge about the geometry of this variety
is not nearly as extensive as Grassmanians or flag varieties. Indeed, the geometry of this variety is much more
intricate than that of the Grassmanian or the complete flag variety.

In this paper we introduce several elementary combinatorial objects that are natural generalizations of involutions in the symmetric
group. We call them \emph{degenerate involutions}, see below. These objects are naturally associated to the geometry of
$\mcal{X}$. On one hand, the goal of this paper is to understand the geometry of $\mcal{X}$ in terms of the combinatorics of these degenerate
involutions and on the other hand we wish to understand the combinatorial properties of degenerate involutions that arise
from geometry of wonderful compactifications.

The variety of complete quadrics $\mathcal{X}_{n}$ is the wonderful compactification of the homogeneous space
$SL_{n}/SO_{n}$, for $n \geq 2$, in the sense of De Concini and Procesi (see \cite{DP83}). The $SL_{n}$-orbits in $\mcal{X}_{n}$ are naturally indexed by
compositions of $n$ and for a fixed composition $\mu$ the $SL_{n}$-orbit $\mcal{O}^{\mu}$ admits a finer decomposition
into Borel orbits. To fix ideas, we work with the Borel subgroup $B$ of the upper triangular matrices in $SL_{n}$ and
the maximal torus $T \subset B$ of diagonal matrices in $SL_{n}$.  The Borel orbits are parametrized by combinatorial
objects called \emph{$\mu$-involutions}. Roughly, these are permutations of $1,2,\dots,n$ that, when subdivided into strings whose lengths are given by the parts of
$\mu$, have each string represent an involution of its alphabet. A \emph{degenerate involution} of length $n$ is just
a $\mu$-involution for some specified composition $\mu$ of $n$.

Extrapolating from the observation that 
the Bruhat order on the symmetric group $S_n$ can be identified with the inclusion order on Schubert varieties in the complete 
flag variety, we define an analogue of Bruhat order on the set of degenerate involutions. Namely, denoting the $B$-orbit
corresponding to the degenerate involution $\pi$ by $\mscr{X}^{\pi}$, we introduce the ordering
\begin{equation}\label{E:bruhat-order-def}
\pi \leq \pi' \mtxt{ if and only if } \mathscr{X}^{\pi} \subseteq \mathscr{X}^{\pi'}.
\end{equation}
One of our main results in this paper is to gain an understanding of the relationship between the induced Bruhat order from
$S_n$ on $\mu$-involutions and the geometric Bruhat order that is defined by the above ordering. 

The next combinatorial object that we introduce is the notion of a \emph{barred permutation}. Barred permutations
parametrize the finitely many torus fixed points of $\mcal{X}_{n}$. Since each $B$-orbit has at most one torus fixed point,
barred permutations can be thought of as certain degenerate involutions and we characterize these degenerate
involutions in \cref{L:TfixedptinBorbit}.

Further, we use these combinatorial objects to study a Bia{\l}ynicki-Birula decomposition of $\mathcal{X}_{n}$. The structure of Bia{\l}ynicki-Birula
cells for spherical varieties is studied in various degrees of generality: for smooth projective spherical 
varieties, by Brion and Luna~\cite{BrionLuna87}; for wonderful compactifications of symmetric varieties, by 
De Concini and Springer~\cite{DS85} and for the particular case of complete quadrics, by Strickland~\cite{Strickland86}.  We
construct two combinatorial maps, $\sigma$ and $\tau$, on the set of all $B$-orbits in a Bia{\l}ynicki-Birula
cell. Given a $B$-orbit associated to a degenerate involution $\pi$,  $\tau(\pi)$ is the degenerate permutation
parametrizing the unique $T$-fixed point in the cell containing the $B$-orbit of $\pi$, see Proposition \ref{prop:quadric-limit}. Conversely, given a 
barred permutation $\gamma$, $\sigma(\gamma)$ provides the degenerate involution corresponding to the $B$-orbit which is dense
in the cell flowing to the $T$-fixed point parametrized by $\gamma$, see Proposition \ref{prop:dense-B-orbit}.

Building on the ideas of Richardson-Springer~\cite{RS90}, Timashev~\cite{Timashev94} and Brion~\cite{Brion98} we study the
Bruhat order described in \cref{E:bruhat-order-def} above. Roughly speaking, the $\WW$-set (resp. the reverse $\WW$-set) of a $B$-orbit $\mscr{Y}$
in $\mcal{X}$, denoted by $\WW(\mscr{Y})$ (resp. $\WW^{-1}(\mscr{Y})$), is the set of Weyl group elements which encode
the saturated chains in the weak-order starting at $\mscr{Y}$ and terminating at the dense $B$-orbit of its
$SL_{n}$-orbit (resp. starting at the closed $B$-orbit its $SL_{n}$-orbit and ending at $\mscr{Y}$). The Richardson-Springer monoid, which is
a natural generalization of the Weyl group, acts on the set of
$B$-orbits contained in $G$-orbit. We denote the action of this monoid by $\star$ below; and $L_{\mu}(\cdot)$ denote the length function on the poset of Bruhat cells
contained in the $SL_{n}$-orbit associated to a composition $\mu$. We have the following theorem.
\begin{thm}[Theorem \ref{thm:Bruhat}]
Let $\pi$ be a $\mu$-involution and $\rho$ be a $\nu$-involution. Then $\rho$ covers $\pi$ in Bruhat order if and only if one of the following holds:
\begin{enumerate}[label={(\roman*)}]
\item $\mu$ is covered by $\nu$ in the refinement ordering (see Definition \ref{E:refinement-order}) and
  $\WW(\pi) \subset \WW(\rho)$.

\item\label{I:T-2} The compositions $\nu = \mu$. Moreover, there exist a simple reflection $s_{\alpha}$  and an element $\varpi \in W$ such that
  \begin{enumerate}
  \item\label{I:T-21}
    $L_{\mu}(\pi) - L_{\mu}(\varpi \star \pi) = L_{\mu}(\rho) - L_{\mu}( \varpi \star \rho)=\ell(\varpi)$;
  \item\label{I:T-22} $s_{\alpha} \cdot (\varpi \cdot \pi) = \varpi \cdot \rho $ (equivalently, $s_{\alpha} \star (\varpi \star \rho) = \varpi \star \pi$);
  \item \label{I:T-23}$ s_{\alpha}\WW^{-1}(\varpi \star \pi) \cap \WW^{-1}(\varpi \star \rho) \neq \emptyset$ where
    $s_{\alpha}\WW^{-1}(\varpi \star \pi)$ is the translation by group action i.e.  $\{ s_{\alpha} w \in W | w \in
    \WW^{-1}(\varpi \star \pi) \}$.
  \end{enumerate}
\end{enumerate}
\end{thm}

When $\mu = (1,1,\dots,1)$, the degenerate involutions are identified 
with permutations and when $\mu = (n)$, then the degenerate involutions are identified with ordinary involutions
\cite{RS94}. In both these extreme cases, the restriction of the Bruhat order to $\mu$-involutions coincides with the opposite of the usual 
Bruhat order on permutations. However, we show that the same property does not hold for all $\mu$ in general.

We now describe the contents of the paper. In Section \ref{S:Preliminaries} we introduce notations and basic
constructions that are used freely throughout the paper.  In Section \ref{S:orbits} we introduce $\mu$-involutions that
provide the enumerate $B$-orbits of $\mcal{X}_{n}$. We also discuss the length functions on the Bruhat poset and action of
Richardson-Springer monoid on the poset of $B$-orbits.
In the subsequent section, Section \ref{S:Barredpermutations}, we introduce barred permutations and their basic
properties. One of our initial goals was to use Brion's presentation of equivariant Chow rings and equivariant formality
of smooth projective varieties to deduce a presentation of the cohomology ring of complete
quadrics. It has been studied using intricate geometric arguments in \cite{DGMP88}. We could only achieve partial
success and our results are described in \cref{SS:coh}. The proof of \cref{thm:Bruhat} is presented in 
\cref{S:Bruhat}. Finally in \cref{S:cells-and-bars}, we study the Bia{\l}nicki-Birula decomposition of $\mcal{X}_{n}$.

\subsection*{Acknowledgment}

The first author was supported by a postdoctoral fellowship funded by the Skirball Foundation via the Center for Advanced
Studies in Mathematics at Ben-Gurion University of the Negev during the preparation of this work.  

%%% Local Variables:
%%% mode: latex
%%% TeX-master: "MC-0"
%%% End:

%                 +++ Ending file MC-0-intro.tex +++   
%------------------------------------------------------------------------------------------------------------------------------------------------

%                 +++ Starting file MC-0-sec2.tex +++   

 % Prelim - Section 2    [[[
\section{Preliminaries}\label{S:Preliminaries}
In this section we will introduce the notation that will be used throughout this paper. We also recall some background
material along the way. 

\subsection*{Notations}
We will use the structure theory of the algebraic group $SL_{n}$ over a field $k$. The  
 Borel subgroup $B$ will be the subgroup of upper triangular matrices, the maximal torus $T$ the subgroup of
diagonal matrices and the Weyl group $W=$ the group of permutation matrices $S_{n}$. The corresponding root system will
be denoted by $\Phi$, positive (resp. negative roots) by $\Phi^{+}$ (resp., $\Phi^{-}$) and the simple roots by
$\Delta$. Given any standard parabolic $P$ (resp. the unipotent radical $U_{P}$ ) the opposite parabolic will be denoted
by $P^{-}$ (resp. unipotent radical of $P^{-}$ will be denoted by $U^{-}_{P}$). The Lie algebras will
be denoted by Gothic fonts e.g., $\mfrk{t}, \mfrk{sl}_{2}$
etc.

We work over a base field $k$. The construction of wonderful compactification is known in all characteristics. The
discussion in \S \ref{SS:coh} and \S \ref{S:cells-and-bars} requires that $ \chr(k) = 0$.

A permutation $\sigma \in S_{n}$ will be represented interchangeably using the \emph{cycle} notation
and the \emph{one-line} notation. For example let $n =5$ and $\sigma \in S_{5}$ be the permutation that interchanges $3$ and $5$ but leaves the other items
unchanged. In cycle notation $\sigma$ will be denoted by $(3,5)$ or (equivalently $(1)(2)(3,5)(4)$) and in one-line notation it will
be denoted by $[12543]$. An \emph{alphabet} for a permutation $\sigma \in S_{n}$ is any ordered subset of natural
numbers, with its natural order, of cardinality $n$ on which
$\sigma$ acts; for example consider alphabets $(1,2,3)$ and $(2,4,8)$ for $S_{3}$, then in one line notation the permutation $[132]$ and $[284]$ are equivalent.     

Finite posets will play an important role in this paper. We recommend the Chapter 3 of \cite{stanley} as a reference. A poset $P$ is \emph{graded} if every maximal chain in $P$ has the same length. A \emph{rank function} on a poset is a function $\mt{rk}: P\rarr \Z_{\geq 0}$ which maps any element $x\in P$ the length a maximal chain from the minimal element to $x$. The \emph{rank} of a graded poset $P$, denoted by $\mt{rk}(P)$, is defined to be the rank of the maximal element.
The posets that we will study in this paper arise from the following geometric situation: a solvable group $B$ acts on a
projective variety with finitely many orbits, and the poset on the set of orbits is generated by the inclusion order on
the closures. It is true that in general such posets are always graded with the minor caveat that there may be more than one minimal 
element, see Exercise 8.9.12 of \cite{RennerBook}. This will not be an issue for the posets we consider in this paper
and hence we will completely ignore it. 

The Bruhat-Chevalley (BC) ordering on symmetric groups (or more generally any Coxeter group) is well known, see
\cite[Chapter 2]{bjrbrn}. The rank function for this poset structure is  
called the \emph{length function} and is denoted by $\ell(-)$. The length of a permutation $\sigma \in S_{n}$ is defined by 
\[ \ell(\sigma) \defeq \mtxt{number of inversion of } \sigma \] where an \emph{inversion} is an ordered pair $(i,j)$
such that $1 \leq i < j \leq n$ and $\sigma(i) > \sigma(j)$. Geometrically this is related to the poset structure of the
$B$-orbits in  the flag variety $SL_{n}/B$.  More precisely, the $B$-orbits are indexed by $ S_{n}$, and we say two
orbits $\mscr{O}^{\sigma} \leq \mscr{O}^{\tau}$ if and only if $\mscr{O}^{\tau}$ is in the topological closure of $\mscr{O}^{\sigma}$. This poset structure (which exists more generally for any spherical
variety) is called the Bruhat ordering. The BC ordering is opposite of the Bruhat ordering \footnote{The confusing nomenclature is deeply entrenched in the
  literature. We will use BC ordering and Bruhat ordering to avoid confusion.}. For general Weyl
groups the \emph{Bruhat decomposition} theorem relates the BC ordering and the Bruhat ordering. 

We consider another case. An involution is an element of $S_n$ of order $\leq 2$. We denote by $\mathcal{I}_n$ the set of involutions in
$S_n$. The restriction of the BC ordering on $S_n$ induces an ordering on $\mcal{I}_{n}$. We call this the BC ordering
on involutions $\mcal{I}_{n}$. This ordering is graded, but somewhat surprisingly, with a different rank function. The rank function, discovered by  
Incitti (see Theorem 5.2 \cite{Incitti04}), is explicitly given by
\begin{equation}\label{eqn:invlength}
      L(\pi) \defeq  \frac{\ell(\pi)+\mt{exc}(\pi)}{2}, \; \mtxt{ for } \pi \in \mcal{I}_{n},
\end{equation}
where $\mt{exc}(w)$ is the {\em exceedance} of $w \in S_n$. It is defined by
\[
\mt{exc}(w) := \# \{ i \in [n] :\ w(i) > i \}.
\]
The exceedance of an involution is the number of 2-cycles that appear in its cycle decomposition.

Suppose an element $\pi \in \mcal{I}_{n}$ has a a cycle decomposition 
\[\underbrace{(a_1,b_1), \cdots, (a_k,
b_k)}_{\mtxt{two-cycles}}\overbrace{c_1, \cdots, c_m}^{\mtxt{one-cycles}}.\] We  associate to $\pi$ a quadric hypersurface in $\mbb{P}^{n-1}$ given by   
\[Q_{\pi} := x_{a_1} x_{b_1} + \cdots + x_{a_k} x_{b_k} + x_{c_1}^2 + \cdots + x_{c_m}^2.\]

Let $V \defeq k^{n}$ denote the standard representation of $SL_{n}$ and $V^{\vee}$ denote the dual space. The collection
of all quadric hyper-surfaces in $V$ is identified
with the representation $\Sym^{2}(V^\vee)$. One can study the Bruhat order induced on the Borel orbits of
$\mbb{P}(\Sym^{2}(V^{\vee}))$. It follows from the work of Richardson and Springer (see \cite{RS94}) that 
the BC ordering on $\mcal{I}_{n}$ is again the opposite of the Bruhat order. In particular, the function given by
\cref{eqn:invlength} becomes the co-rank function of the Bruhat order. 

\begin{comment}
\begin{rem}\label{R:atmost3}
The ordering we follow is the \emph{opposite} of the ordering followed by Incitti. In our convention, the proof of Theorem 5.2 of \cite{Incitti04}  shows that that for two involutions $\pi,\pi'$ in  $\mc{I}_n$ if $\pi'$ covers $\pi$, then $\ell(\pi) - \ell(\pi') \leq 3$.
\end{rem}
\end{comment}

\subsection{The variety of complete quadrics}\label{Ss:compq}

Let us assume $n \geq 3$ for simplicity. The variety of complete quadrics, denoted by $\mcal{X}_{n}$ has a long and rich history. To
the best of our knowledge, there are three independent ways to construct this variety. The first two are
algebro-geometric in nature, and the third one is representation theoretic. Roughly speaking, the algebro-geometric
method starts with a simple variety and then repeatedly applies geometric constructions (blow-ups or taking Zariski closure) on this initial
variety (see-below) to arrive at $\mcal{X}_{n}$. On the other hand the representation theoretic construction presents
the variety as a subvariety of a projective space of much bigger dimension. We will use the
representation theoretic construction. We briefly recall the geometric construction for its historical significance.

\subsubsection{Algebro-geometric construction} \label{SS:agq}
Let us denote by $\mcal{X}^{0}$ the space of isomorphism classes of symmetric, non-degenerate quadratic forms on the
$n$-dimensional affine space $k^{n}$ or equivalently non-singular quadric hypersurface in
$\mbb{P}^{n-1}_{k}$. Let $\mbb{P}^{N}$ denote $\mbb{P}(\Hom_{k}(k^{n},k^{n}))$ and identifying a quadratic form with the
associated matrix we have a natural embedding of $\mcal{X}^{0} \subset \mbb{P}^{N}$. Given any quadric $\Q \in
\mcal{X}^{0}$ and any integer $i \leq n-1$ one can define an incidence variety
$\Gamma_{Q} \subset \mbb{P}^{n-1} \times_{k} \gr(i,n)$, where $\gr(i,n)$ is the Grassmanian of $i$-dimensional sub-spaces
of $k^{n}$, called the variety of $i$-dimensional tangent spaces to $Q$. Let $\mathbf{Gr} \defeq
\prod_{i=1}^{n-1}\gr(i,n)$. Then we have constructed a double fibration
\begin{equation}
  \begin{tikzcd}
 &  \mbb{P}^{N} \times_{k} \mathbf{Gr} \arrow[ld, "\pi_{1}"] \arrow[rd, "\pi_{2}"] &  \\
   \mcal{X}^{0} \subset \mbb{P}^{N} &   & \mathbf{Gr} 
  \end{tikzcd}
\end{equation}
 and the variety of complete quadrics is defined as the image $\overline{\pi_{1}^{-1}(\mcal{X}^{0})}$. This construction
 was realized by Tyrrell using \emph{higher adjugates}, see \cite{klth}.

The second construction, due to Vainsencher, starts with $\mcal{X}^{0} \subset \mbb{P}^{N}$ and then realizes
$\mcal{X}_{n}$ as a transform of iterated successive blowups of $\mbb{P}^{N}$ with cleverly chosen centers. This
rather intricate construction has been generalized by many authors; we recommend the article \cite{klth} for a
comprehensive overview and detailed proofs.
 
\subsubsection{Representation-theoretic construction}\label{SS:repthr}

The representation-theoretic construction of $\mcal{X}_{n}$ is a consequence of the more general construction of
\emph{wonderful compactifications} of De Concini and Procesi. In literature, this construction is presented in an
abstract way which handles all Lie-group types uniformly; see \cite{DP83, sprdec, falt}. We will recall the important
parts of this construction for complete quadrics and fix a specific model (all models are $G$-equivariantly
isomorphic). We will closely follow the notation of \cite{sprdec}.

We set $G = SL_{n}(k)$ and an involution $\theta$ on $G$ given by $\theta(\tau) =
(\tau^{-1})^{\mathbf{t}}$, where $\mbf{t}$ denotes the transpose of a matrix. The fixed points $G^{\theta}$ is identified with $SO(n)$. The map $G/G^{\theta} \rarr
\mcal{X}^{0}$ taking $\tau \mapsto \tau \cdot \tau^{\mathbf{t}}$ connects the homogeneous space $G/G^{\theta}$ and the
space $\mcal{X}^{0}$ in \S \ref{SS:agq} above. The space $\mcal{X}_{n}$ is then obtained as an wonderful compactification of
the space $G/G^{\theta}$.

The key point, in the construction of $\mcal{X}_{n}$, is that one can single out a (non-empty!) class of finite dimensional
algebraic representations $\{\mbb{V}_{i} \}$ with the property that there is a non-zero vector
$v_{i} \in \mbb{V}_{i}$ such that the closure of the orbit $G \cdot v_{i} \subset
\mbb{P}(\mbb{V}_{i})$, where $(\stab_{G}(v_{i}) = G^{\theta})$. By a model, we mean fixing such a representation
$\mbb{V}_{i}$ and a spherical vector
$v_{i}$.

Concretely, given $G$ and the involution $\theta$ as above, let $T$ denote the standard maximal torus of $G$ consisting of the diagonal matrices, $B$ denote the upper
triangular matrices; clearly $\theta(B) = B^{-}$. Let $\Phi$ denote the roots of $G$ (with respect to this choice
of Borel subgroup $B$ and $T$), and
$\Delta = \{ \alpha_{1}, \ldots \alpha_{n-1}\}$ the standard simple roots of $\Phi$. Let $W = S_{n}$ denote the Weyl group of $G$.

We let $k^{n}$ denote the standard representation of $G$ and consider $V\defeq \oplus_{i=0}^{n} \Lambda^{i}(k^{n})$ with
the induced representation of $G$. We let $^{\theta}V$ denote the representation of $G$ on $V$ twisted by the automorphism $\theta$. Consider the $k$-vector space $\mbb{V} \defeq
\Hom_{k}(^{\theta}V, V)$ as a representation of $G$. Let $h \in \mbb{V}$ denote the the identity map. We note that $h$
is a spherical vector (invariant under $G^{\theta}$ action). The wonderful compactification $\mcal{X}_{n}$ is
the closure of the $G$-orbit $G \cdot [h]$ in $\mbb{P}(\mbb{V})$. The key properties of the wonderful compactification,
outlined in Theorem \ref{T:VDP} below, rests on the following crucial observation. 

The highest weight vector in $\mbb{V}$ has weight $\rho = \sum_{i=1}^{n-1} 2\alpha_{i}$ and let $pr_{\rho}:
\mbb{P}(\mbb{V}) \rarr k$ denote the projection onto the line spanned by the highest weight vector. The non-vanishing
locus $\mcal{X}_{n} \cap \{ pr_{\rho} \neq 0 \}$ is the affine space $U^{-} \times_{k} \overline{T\cdot h}$ and the toric
variety $\overline{T\cdot h}$ is equivariantly isomorphic to the $(n-1)$ affine space $\mbb{A}^{n-1}$ with $T$ action
given by $t \cdot(v_{1}, \ldots v_{n-1}) = (t^{-2\alpha_{1}}\cdot v_{1}, \ldots, t^{-2\alpha_{n-1}}\cdot v_{n-1})$, and
under this identification the vector $(1, \ldots,1))$ corresponds to $h$. All 
$G$ orbits closures in $\mcal{X}_{n}$ intersect $\overline{T \cdot h}$ along a $T$-stratum in $\mbb{A}^{n-1}$.

More precisely, the $G$-orbit closures in $\mcal{X}_{n}$ are in one-to-one correspondence with subsets of $\Delta$; for any subset $S \subset \Delta$ the corresponding $G$-orbit closure $\mcal{X}^{S}$ fibers over the partial flag
variety $G/P_{S}$ ($P_{S}$ = standard parabolic containing $B$ corresponding to $S$). The $G$-orbit closure
$\mcal{X}^{S}$ intersects $\mbb{A}^{n-1}$ along the toric stratum
\begin{align}\label{E:toric-strata}
\mbb{A}_{S} = \{ (x_{1}, \ldots,x_{n-1}) \in \mbb{A}^{n-1} : x_{i} = 0 \mtxt{ for } i \notin S \}.
\end{align}

Summarizing, the main features of the wonderful compactification $\mcal{X}_{n}$ are outlined below. 
\begin{thm} [See \cite{sprdec, DP83} ] \label{T:VDP}
      The variety $\mathcal{X}_n$ has the following properties. 
 \begin{enumerate} [label={(\roman*i)}]
       \item $\mcal{X}_{n}$ is smooth and projective. 
       \item The complement $\mcal{X}_{n} \setminus \mcal{X}^{0}$ is a union of smooth normal crossing divisors
         $\mcal{X}^{i}$, where $i$ varies over the subsets of $\Delta$, and any $G$-orbit closure $\mcal{X}^{S} \defeq
         \cap_{i \in S} \mcal{X}^{i}$ where $S\subset \Delta$  fibers over the partial flag variety $G/P_{S}$. Note
$P_{S}$ is uniquely determined by $S$ and the requirement $B \subset P_{S}$.

\item The variety $\mcal{X}_{n}$ is uniquely determined by a unique $G$-equivariant isomorphism. 
\end{enumerate}
\end{thm}

\begin{rem} \label{R:Fibstr}
A remarkable (wonderful!) aspect of the construction of $\mcal{X}_{n}$ is the following. Given any subset $S \subset \Delta$, the corresponding
$G$-orbit closure $\mcal{X}^{S}$ and the dense open $G$-orbit $ \mcal{O}^{S}$ in $\mcal{X}^{S}$ fit into a diagram
             \begin{equation} \label{E:Fibr}
                   \begin{tikzcd}
                         \mcal{O}^{S} \arrow[r, "j^{S}"] \arrow[rd] & \mcal{X}^{S} \arrow[d, "\pi_{S}"] \\
                          & G/P_{S}.
                   \end{tikzcd}
             \end{equation}

Let $L_{S}^{ss}$ denote the semi-simplification of Levi-component of $P_{S}$ containing $T$. We have $L_{S}^{ss} = \prod SL_{m_{i}}\;
(\sum m_{i} = n )$ and each the involution $\theta$ on $SL_n$ induces the same involution on each $SL_{m_{i}}$. Let
$\mbb{O}^{S}$ denote the direct product of $SL_{m_{i}}$-homogeneous spaces $\prod_{k}SL_{m_{i}}/SO_{m_{i}}$ and
$\mbb{X}^{S}$ denote the direct product of wonderful compactifications $\prod_{k}\mcal{X}_{m_{i}}$. Then, after
extending the component-wise $L_{S}^{ss}$ actions trivially to unipotent radical of $P_{S}$, we get $P_{S}$ actions on
$\mbb{O}^{S}$ (resp $\mbb{X}^{S}$).  We have unique $G$-equivariant isomorphisms $\mcal{O}^{S} = G \times_{P_{S}}
\mbb{O}^{S}$ and $\mcal{X}^{S} = G \times_{P_{S}}\mbb{X}^{S}$. This is very useful for certain inductive arguments. 

\end{rem}

\begin{rem}\label{R:ComplQ}
   
  In the light of the previous remark, a point in $\mcal{X}_{n}$ can be intuitively thought of a pair $(\mcal{F}^{\bullet},
Q_{\mcal{F}^{\bullet}} )$. Where $\mcal{F}^{\bullet}= \{ V_{0}\subset \ldots \subset V_{k}\}$ is a partial flag variety of $k^{n}$ and
$Q_{\mcal{F}^{\bullet}}$ is a collection of non-degenerate quadric hypersurface in the successive (projective) sub-quotients
$V_{i+1}/V_{i}$. Sometimes, when the flag is clear from the context, we will loosely say $Q_{\pi}$ is a complete
quadric. 
\end{rem}

% BB-STUFF [[[
\subsection{Bia{\l}ynicki-Birula decomposition}\label{S:bbdecomp}

The Bia{\l}ynicki-Birula decomposition (BB decomposition for short) is an important tool for studying algebraic actions
of torus on projective algebraic varieties. We recall a version of the BB-decomposition theorem, which will be
sufficient for our requirements. 

\begin{thm}[Theorem 4.3 \cite{BB73}] \label{T:BB}
      Suppose $X$ is a smooth connected complete variety with an algebraic action of the torus $\mbb{G}_{m}$. Suppose $X$ has
      finitely many torus fixed points $\left\{ x_{1}, \ldots , x_{r}\right\}$.  Then there exists locally closed (in the
      Zariski topology) $\mbb{G}_{m}$-invariant subschemes $X_{i}^{+}$ satisfying the following properties.

      \begin{itemize}[label={$\circ$}]
    
\item The schemes  $X_{i}^{+}$ partition $X$, i.e. $X = \cup_{i=1}^{r} X_{i}^{+}$ and $X^{+}_{i} \cap X_{j}^{+} = \emptyset$. 

\item The subschemes $ X_{i}^{+}$ are (locally closed) affine spaces and for each index $i$, we have  $x_{i} \in X_{i}^{+}$. 
      \end{itemize}
\end{thm}

We will call the  affine spaces $X_{i}^{+}$ the BB-cell attached to $x_{i}$.  

The BB-decomposition should be seen as an algebraic version of Morse stratification and it has many important and
similar consequences. Perhaps the most important one is that the partitions $\left\{ X^{+}_{i}\right\}$ provide
a topological filtration of $X$ into affine cells and hence the classes of the closures $\left \{ \overline{X^{+}_{i}}
\right \}$ form a basis in $H_{\ast}(X; \mbb{Z})$. Likewise, the Poincar\'{e} duals of these classes form a basis in
cohomology. 

 %     ]]]          

%%% Local Variables:
%%% mode: latex
%%% TeX-master: "MC-0"
%%% End:

%                 +++ Ending file MC-0-sec2.tex +++   
%------------------------------------------------------------------------------------------------------------------------------------------------

%                 +++ Starting file MC-0-sec3.tex +++   

\section{Parametrization of $SL_{n}$ Orbits in $\mcal{X}_{n}$} \label{S:orbits}
In this section we describe combinatorial indexing of $SL_{n}$-orbits and  $B$-orbits in  $\mcal{X}_{n}$.

\begin{defn}\label{D:refinement-order}
      A \emph{composition} of a positive integer $n$ is an ordered sequence $\mu=(\mu_1,\dots, \mu_k)$ of positive integers that sum to $n$. 
      The elements of the sequence $\mu_{i}$ are called the \emph{parts} of $\mu$.
 \end{defn}
 
 The compositions of $n$ corresponds to subsets of $\left\{1,3,\ldots, n-1 \right\} $ via the bijection
\begin{equation}\label{eqn:mutoI}
      \mu =(\mu_1,\dots, \mu_k) \longleftrightarrow I(\mu) := \left\{ 1,2,\ldots, n-1 \right\} \setminus  \left \{ \mu_1, \mu_1+\mu_2,\dots, \mu_1+\cdots + \mu_{k-1} \right \}.
\end{equation}
This correspondence gives a simple way to describe the refinement order on compositions of $n$. The \emph{refinement order} on
compositions of $n$ is defined by: $\mu \preceq \nu$ if and only if $I(\mu) \subseteq I(\nu)$. Informally, $\mu$ refines $\nu$ if $\mu$ can be obtained from $\nu$ by subdividing its parts. 
It follows that the most refined composition, $(1,1,\dots,1)$ is the unique minimal element of this ordering and the trivial
composition $(n)$ is the maximal element.

% \begin{rem}
% In the context of the variety $\mcal{X}_{n}$ one should really think of a composition in terms of the set $I(\mu)$. The
% latter is nothing but a way to keep track of a subset of simple roots of $SL_{n}$.
% \end{rem}

We recall from \cref{T:VDP} that the closed $G$-orbits of $\mcal{X}_{n}$ are in bijective correspondence with the subsets of simple roots of
$G$. This allows us to label the $G$-orbit closures by compositions. A $G$-orbit closure $\mcal{X}^{\mu}$, corresponding to a composition $\mu$, contains a unique open $G$-orbit denoted by $\mcal{O}^{\mu}$. 
The refinement order on compositions correspond to the inclusion order on the orbit-closures:   
\begin{equation} \label{E:refinement-order}\mathcal{X}^{\mu} \subseteq \mathcal{X}^{\nu} \Longleftrightarrow \mu \preceq
  \nu.
\end{equation}
The maximal element, with this poset structure, 
corresponds to the whole space $\mcal{X}_{n}$ and the minimal element corresponds to the variety of complete flags in $k^n$. 

A $G$-orbit $\mcal{O}^{\mu}$ is a union of finitely many $B$-orbits. We denote $B$-orbits of $\mcal{X}_{n}$ by script
letters $\mscr{O}$ (resp. $\mscr{X}$) to distinguish from $G$-orbits $\mcal{O}$ (resp. $\mcal{X}$). It turns out that
$B$-orbit closures (and hence their orbits) in $\mcal{O}^{\mu}$ are parametrized by combinatorial objects called \emph{$\mu$-involutions}.

\begin{defn}
A \emph{$\mu$-involution} $\pi$ is a permutation of the set $[n]$, which when written in one-line notation and 
partitioned into strings of size given by $\mu$, that is $\pi = [\pi_1 | \pi_2 | \dots | \pi_k]$ with $\pi_j$ a string of 
length $\mu_j$, has the property that each $\pi_j$ is an involution when viewed as the one-line notation of a 
permutation of its alphabet. We will sometimes refer to the sub-strings $\pi_{i}$ as a the components of $\pi$. 
\end{defn}

\begin{exmple}\label{E:muex}
For example, $\pi = [26|8351|7|94]$ is a $(2,4,1,2)$-involution and the string $8351$ is viewed as one-line notation 
for the involution $(1,8)(3)(5)$ of its alphabet. (We adopt the non-standard convention of including one-cycles when 
writing a permutation in cycle notation, since we have to keep track of what alphabet is being permuted when 
working with $\mu$-involutions.)
\end{exmple}

\begin{defn}\label{def:mu-inv}
      Suppose $\mu = (\mu_{1}, \ldots, \mu_{k})$ is a composition of $n$. Let $\pi = [\pi_1 | \pi_2 | \dots | \pi_k]$ be a $\mu$-involution. Let $\mcal{A}_{j} \subset [n]$ denote the alphabet of the permutation $\pi_{j}$. 
      A \emph{distinguished complete quadric} $Q_{\pi}$ associated to $\pi$  is the complete quadric $Q_{\pi}$ given by the following data. 
\begin{enumerate}[label={(\roman*)}]

\item A partial flag 
  \[ \mcal{F}_{\pi} :  0 = V_0 \subset V_1 \subset V_2 \subset \dots \subset V_{k-1} \subset V_k = k^{n} \]
  where $V_{j}$ is spanned by the standard basis vectors $e_{a_{i}}$ for $a_{i} \in \mcal{A}_{i}$ with $i \leq j$. Note that $\dim V_{j} = \mu_1 + \mu_2 + \cdots + \mu_j$.
     
\item On each successive quotient $V_{j}/V_{j-1} $ a non-degenerate quadric $Q_{\pi_{j}}$ is given by the recipe: if $\mcal{A}_{j}$ is
      the alphabet underlying the involution $\pi_{j}$ and suppose $\pi_{j}$ has a cycle decomposition (in this alphabet) of the
      form $(a_1,b_1)\ldots(a_s, b_s)(c_1)\ldots(c_t)$ then $Q_{\pi_{j}}$ is the quadric
$x_{a_1} x_{b_1} + \cdots + x_{a_s} x_{b_s} + x_{c_1}^2 + \dots x_{c_t}^2$.
\end{enumerate}
\end{defn}

\begin{exmple}
      For example consider the $\mu$-involution $\pi$  as in \cref{E:muex} above. Then the flag $\mcal{F}_{\pi}$ 
      is given by $$ 0 \subset V_{26} \subset V_{123568} \subset V_{1235678} \subset V$$ and the 
associated sequence of non-degenerate quadrics is $x_2^2 + x_6^2,  x_1x_8 + x_3^2 + x_5^2, x_7^2, x_4 x_9$.

\end{exmple}

%  ]]]]

% WEAK ORDER ON INVOL    [[[

\subsection{Weak Order for $\mu$-involutions}

The Richardson-Springer (RS) monoid, associated to $S_n$ and denoted by $\mcal{M}(S_n)$, is the monoid generated
by elements  \( \langle{s_1, \dots, s_{n-1}} \rangle \) subject to the relations
\begin{align*}
  s_i^2 = s_i & \; \mtxt{for all } i, \\
  s_i \cdot s_j = s_j \cdot s_i & \mtxt{  if }|i - j| > 1, \\
  s_i \cdot s_{i+1} \cdot s_i = s_{i+1} \cdot s_i \cdot s_{i+1} &  \mtxt{ for } 1 \leq i < n - 1.
\end{align*}

The set theoretic mapping taking the transposition $(i, i+1) \in S_{n}$ to $s_{i} \in \mcal{M}(S_{n})$ extends to a well
defined map to all of $S_{n}$. In other words, if  $w \in S_{n}$ admits a reduced expression $s_{i_{1}} \ldots
s_{i_{k}}$ then the corresponding element $s_{i_{1}} \cdot \ldots \cdot s_{i_{k}} \in \mcal{M}(S_{n})$ is independent of
the choice of reduced expression of $w$. The action of $S_{n}$ on a set and the action of $\mcal{M}(S_{n})$ on the same
set are quite different because the set theoretic bijection between $S_{n}$ and $\mcal{M}(S_{n})$ is not a monoid
morphism. In the sequel, the intended action will be clear from the context.

There is a natural action of the Richardson-Springer monoid of $S_n$ on the set of all $B$-orbits
and consequently on the set of all $\mu$-involutions, see \cite{RS90} for details. In the case
where $\pi$ is an ordinary involution of $S_n$ (i.e., when $\mu = (n)$ or equivalently a $B$-orbit in the open $G$-orbit
of $\mcal{X}_{n}$), the action of the generator $s_{i}$ corresponding to the simple 
transposition $(i,i+1)$ is explicitly given by

\begin{equation}\label{eqn:RSaction}
s_i \cdot \pi = \begin{cases} s_i \pi s_i & \text{if $\ell(s_i \pi s_i) = \ell(\pi) - 2$} \\
s_i \pi & \text{if $s_i \pi s_i = \pi$ and $\ell(s_i \pi) = \ell(\pi) - 1$}
\\ \pi & \text{otherwise} \end{cases},
\end{equation}
where the multiplication in the right-hand-side is the group multiplication in $S_{n}$.

If $\pi = [\pi_1 | \pi_2 | \dots | \pi_k]$ is a general $\mu$-involution, then the action of $s_{i}$ is as follows. 
\begin{enumerate}[label={(Case \roman*)}]

\item If there is a sub-string, say  $\pi_r$ of $\pi$, whose alphabet contains the letters $i, i+1$ then
  \[s_i \cdot \pi = s_{i} \cdot [\pi_{1}| \ldots| \pi_{k}] \defeq [\pi_{1}| \pi_{2}| \ldots | s_{i} \cdot \pi_{r}|
    \ldots | \pi_{k}]\] where $s_i \cdot \pi_r$ is defined by \cref{eqn:RSaction} considering each $\pi_{r}$ as an involution of its alphabet. 

\item If no sub-string of $\pi$ is of the above form then 
  
\[
  s_{i} \cdot \pi \defeq \begin{cases} \mtxt{ interchange letters } i \mtxt{ and } i+1 &  \mtxt{ if } i+1 \mtxt{ precedes } i \mtxt{ in } \pi \\
    \pi & \mtxt{ otherwise }
  \end{cases}
\]
\end{enumerate}

The action $w \cdot \pi$ of an arbitrary $w \in \mcal{M}(S_n)$ and arbitrary $\mu$-involution $\pi$ is defined
recursively: if $w = s_{i_1} s_{i_2} \cdots s_{i_\ell}$ be any reduced expression then
\[ w \cdot \pi \defeq s_{i_1} \cdot ( s_{i_2} \cdot \cdots ( s_{i_\ell} \cdot \pi) \cdots ).\]

\begin{defn}
Given any  $\mu$-involution
$\pi = [\pi_1 | \pi_2 | \dots | \pi_k]$. Consider the length function $L_{\mu}(\pi)$ is defined by the formula
\begin{equation}\label{eqn:mu-invlength}
L_{\mu}(\pi) := \ell(w(\pi)) + \sum_{i=1}^k L(\pi_i),
\end{equation}
where $w(\pi)$ is the permutation obtained by rearranging the elements in each string
$\pi_i$ in increasing order and $L(\pi_i)$ is the length of the corresponding involution $\pi_i$ as
defined by \eqref{eqn:invlength}. 
\end{defn}

\begin{exmple}
If $\pi = [5326 | 41]$ a $(4,2)$-involution. Then $w(\pi) = 235614$ and
$L_{(4,2)}(\pi) = 6 + 2 + 1 = 9$.
\end{exmple}

We define the weak order on two $\mu$-involutions. The poset gives a partial  order between the $B$-orbits appearing in
a fixed $G$-orbit. 

 \begin{defn} \label{Def:wkordr}
   The \emph{weak order} on two $\mu$-involutions $\pi$ and $\rho$ is given by \[\pi \leq_{W} \rho \mtxt{ if and only if  } \rho = w \cdot \pi\]
       for some element $w$ in the RS-monoid $\mcal{M}(S_n)$. 
 \end{defn}
 
The covering relations in weak order are labeled by simple roots. The associated poset has a
maximal and minimal element (denoted by $\min$ and $\max$ respectively). The minimum and maximum elements, in the weak
order, also admit explicit descriptions:  $\min$ (resp, $\max$) denote the string $n \dots 2 1$ (resp., $1 2 \dots n$) partitioned according to $\mu$. 

Starting from any element $\pi$ one can
construct maximal chains recursively by successively picking
simple transpositions and letting them act on a previous element of the chain.  Roughly speaking this leads us to the
idea of a  $\WW$-set of $\pi$, denoted by $\WW(\pi)$. It is the set of all
elements $w \in S_{n}$ such that the $ w\cdot \pi = \max$ and moreover $L_{\mu}(\max) - L_{\mu}(\pi) = \ell(w)$ forms a
chain in the weak order poset. For our purposes, we will consider a slight generalization of this notion, see
\cref{def:w-set} below.

\begin{defn} [see \cite{CanJoyce13, CanJoyceWyser1, CanJoyceWyser2}]~\label{def:w-set}
Let $\pi, \rho$ be two $\mu$-involutions. The $\WW$-set of the pair $(\pi, \rho)$ is the subset of $S_{n}$  defined by \[\WW(\pi, \rho) := \{ w \in
S_n : w \cdot \pi = \rho \text{ and } \ell(w) = L_{\mu}(\pi) - L_{\mu}(\rho) \}.\] The $\WW$-set $\WW(\pi, \max)$ will be
denoted by $\WW(\pi)$. In this case the  \emph{reverse $\WW$-set} of $\pi$, denoted by $\WW^{-1}(\pi)$ is the $\WW$-set $\WW(\min, \pi)$.
\end{defn}

Note the $\WW(\pi, \rho) \neq \emptyset$ if and only if $\pi \leq \rho$ in weak order. 

The function $L_{\mu}(\min) - L_{\mu}(\pi)$ defines the rank function on the weak order poset of $\mu$-involutions. It
is a generalization of the order on ordinary involutions defined in  see \cref{eqn:invlength}.

\begin{rem} Some remarks about the Definition \ref{def:w-set} are in order.

 \begin{enumerate} 
 % \item \label{I:opposite-order} Consider the set of $\mu$-involutions but with the opposite order of the weak
 %   ordering. Then the reverse $\WW$-set $\WW^{-1}(\pi)$  with respect to weak order is the $\WW$-set of $\pi$ with
 %   respect to the opposite order. 

\item The inverse of an element in $W(\pi, \rho)$ in Definition \ref{def:w-set} is referred to as an
\emph{atom} and the set of atoms are described concretely by Theorems 5.10 and 5.11 in
\cite{HMP}.

\item The $\WW$-sets have geometric significance. This has been thoroughly investigated by Brion in \cite{Brion98}.  

\end{enumerate}
\end{rem}

\section{Barred permutations}\label{S:Barredpermutations}
In this section we introduce \emph{barred permutations} which parametrize the torus fixed points in $\mcal{X}_{n}$.

\begin{defn}\label{D:special subset}
A composition $\mu$ of $n$ is called \emph{special} if every part $\mu_{i}$ of $\mu$ has length at-most $2$. Equivalently, $\mu$ is special
if the associated subset $I(\mu)$ (see \cref{eqn:mutoI}) does not contain any consecutive integers. 
\end{defn}
Note that the refinement of a special composition is also a special composition.

\begin{defn}
    Let $\mu=[\mu_{1}|\mu_{2}| \ldots|\mu_{k}]$ be a special composition. A $\mu$-involution $\pi=[\pi_{1}|\ldots| \pi_{k}]$ is called a
    \emph{barred permutation} if whenever $\mu_{k}$ has length two the string $\pi_{k}$ is of the form $\pi_{k} = ji$
    with $j > i$. For example $[1|32]$ is a barred permutation but $[1|23]$ is not. 
\end{defn}

The set of all barred permutations associated to a composition of type $\mu$ will be denoted by $\mcal{B}_{\mu}$ and we
let  $\mcal{B}_{n}$ denote all possible barred permutations on $[n]$.

In particular, given any special composition $\mu$ of
$n$, the ordered sequence $(1,2, \ldots, n)$ corresponds to a unique barred permutation in $\mcal{B}_{\mu}$. We call it the \emph{special element} of $\mcal{B}_{\mu}$.   

\begin{lema} \label{L:barrsize}
Let $b_n \defeq \#\mathcal{B}_n$. Then sequence $b_n$ satisfies the following recurrence relation. 
\begin{align}\label{A:basicrecurrence0}
b_{n+1} = {n+1 \choose 2 } b_{n-1} + (n+1) b_{n} \text{ for } n\geq 1,
\end{align}
and the initial conditions $b_0 = b_1 = 1$.
\end{lema}
\begin{proof}
Let $\pi = [\pi_1 | \cdots | \pi_{k-1} | \pi_k]$ be a barred permutation on $[n+1]$. We count possibilities for $\pi$ according to its last string $\pi_k$. The first term in the recurrence counts the number of barred permutations where the length of $\pi_k$ is $2$ and the second term in the recurrence counts the number of barred permutations where the length of $\pi_k$ is $1$.
\end{proof}

\begin{prop}\label{T:formula}
The exponential generating series $F(x)=\sum_{n\geq 0} \frac{b_n}{n!} x^n$ for the number of barred permutations of length $n$ is given by
\[
F(x)=  \frac{1}{1-x-x^2/2} = \sum_{n=0}^{\infty} \frac{1}{\sqrt{3}} \left( \frac{(1+\sqrt{3})^{n+1} - (1 - \sqrt{3})^{n+1}}{2^{n+1}} \right) x^n.
\]
Hence, the number of barred permutations (equivalently the number of $T$-fixed points in $\mathcal{X}_n$) is
\[
b_n= \frac{n!}{2^n} \sum_{i=0}^{\lfloor n/2 \rfloor} {n+1 \choose 2i+1} 3^i.
\]
\end{prop}

\begin{proof}
   In the light of \cref{L:barrsize}, we substitute $a_k := b_k / k!$ in the \cref{A:basicrecurrence0} above. This leads to the new
     linear recurrence relation
\begin{align}\label{A:aks}
a_k = a_{k-1} + \frac{1}{2} a_{k-2} \text{ for } k \geq 2
\end{align}
with initial conditions $a_0 = a_1 = 1$.

The proposition follows from solving the resulting linear recurrence relation.  
\end{proof}

\begin{prop}\label{L:TfixedptinBorbit}
Let $\mu =(\mu_1,\dots, \mu_k)$ be a given composition and $\pi = [\pi_1 | \cdots | \pi_k ]$ a $\mu$-involution. Then
the $B$-orbit $\mscr{O}^{\pi}$ contains a $T$-fixed point if and only if $\pi$ is a barred permutation. Moreover, in
this case the torus fixed point is the distinguished quadric $Q_{\pi}$.
\end{prop}

\begin{proof}
    It follows from the work of Strickland, see \cite{Strickland86}, that each component $\mu_{i} \leq 2$.
    
      Let $Q_{\pi} \defeq (\mcal{F}_{\pi}^{\bullet}, Q_{\mcal{F}^{\bullet}})$ denote the complete quadric associated to
      $\pi$, see \cref{R:ComplQ} for the notation. We set $\mcal{F}^{\bullet} = 0 \subset V_{1} \subset \ldots \subset V_{k} =
      k^{n}$ and let $Q_{\mcal{F}^{\bullet}}^{i}$ denote the quadric hypersurface on the projectivized sub-quotient
      $V_{i}/V_{i-1}$.   The projection map from the open $SL_{n}$-orbit $\mcal{O}^{\mu}$
      containing $Q_{\pi}$ to the partial flag-variety $G/P_{\mu}$ is $SL_{n}$-equivariant and hence $Q_{\pi}$ is
      $T$-fixed if and only if its projection $\mcal{F}^{\bullet}_{\pi}$ is $T$-fixed and the point $Q_{\mcal{F}^{\pi}}$
      is fixed by the induced $T$-action on the fiber.

      The description of $T$-fixed flags in a partial flag variety is well-known -- these correspond to permutation of the
      standard flag. It is clear from the description of the fibers of $\mcal{O}^{\mu} \rarr G/P_{\mu}$ that the $T$-action
      on the fiber  $Q_{\mcal{F}^{\bullet}}$ is given by diagonal action on each factor $Q_{\mcal{F}^{\bullet}}^{i}$.
      The induced $T$-action on each factor $Q_{\mcal{F}^{\bullet}}^{i}$ is given by the action of diagonal matrices on symmetric
      $\mu_{i} \times \mu_{i} $ matrices associated to  $Q^{i}_{\mcal{F}^{\bullet}}$ (explicitly, $ D \mapsto D^{\mbf{t}}\cdot[Q^{i}_{\mcal{F}^{\bullet}}]\cdot D$).

      We have $\mu_{i} = \dim_{k}(V_{i}/V_{i-1}) \leq 2$. When $\mu_{i} = 1$ the only $T$-invariant quadric hypersurface is given by
      $x^{2}$; when $\mu_{i}=2$, direct computation shows that the $T$-invariant quadric hypersurface is $xy$ (and not
      $x^{2}+ y^{2}$). In other words, when $\mu_{i}= 2$, the vector space  $V_{i}/V_{i-1}$ is generated by the projection of standard
      basis vectors $e_{\alpha_{i}}, e_{\beta_{i}}$ for $\alpha_{i} < \beta_{i}$ and the factor $Q^{i}_{\mcal{F^{\bullet}}}$ corresponds to the
      involution which, in one-line notation, must be $[\beta_{i}, \alpha_{i}]$. This shows that indeed if $Q_{\pi}$ is
      $T$-fixed then $\pi$ must be a barred permutation.

      Conversely, if $\pi$ is a barred permutation then the distinguished quadric associated to $Q_{\pi}$ is evidently
      $T$-fixed. This proves the proposition.   

\end{proof}

\subsubsection{Weyl group action}\label{Ss:wga}

Let $\mu$ be any special composition of $n$. Given any $\sigma \in S_{n}$ consider the automorphism $\sigma:
\mcal{B}_{\mu} \rarr \mcal{B}_{\mu}$ defined on elements by associating $\pi \mapsto \sigma(\pi)$ where $\sigma(\pi)$ is
obtained in the following way.
\begin{itemize}
          \item Remove all bars from $\pi$ and consider the resulting ordered string $\pi^{\prime}$.
          \item Apply the permutation $\sigma$ to the string $\pi^{\prime}$ and consider the resulting string $\sigma(\pi^{\prime})$.
          \item Reintroduce the bars on the ordered string  $\sigma(\pi^{\prime})$, making it into a $\mu$-involution,  and adjust length two
strings, if necessary, to get a barred permutation. 

\end{itemize}

In the light of \cref{L:TfixedptinBorbit} the following lemma is immediate. 
\begin{lema}\label{L:WGrpAct}

    The automorphisms $\sigma$ define an action of $S_{n}$ on $\mcal{B}_{\mu}$. Moreover, let  $S_{\mu}$ denote the 
 parabolic subgroup of $S_{n}$, corresponding to the canonical map of the $G$-orbit $\mcal{X}^{\mu}$ to the partial flag
 variety $G/P_{I(\mu)}$. Then we have an $S_{n}$-equivariant bijection between  $S/S_{\mu}$ and $\mcal{B}_{\mu}$ which sends
 $[W_{\mu}]$ to the special element of $\mcal{B}_{\mu}$. 

\end{lema}

\subsubsection{Subdivision operator}

Given any integer $1 \leq i < j \leq n$ we will associate a subdivision operator $\sd_{ji}: \mcal{B}_{n} \rarr \mcal{B}_{n}$
as follows. 

\begin{defn}\label{D:subdivison}
  Suppose $\pi=[\pi_{1}| \ldots| \pi_{k}]$ is any barred permutation. Then  
  \begin{equation}
      \sd_{ji}(\pi) = \begin{cases}
          \pi & \mtxt{ if no component } \pi_{r} \mtxt{ is of the form } ji \\
[\pi_{1}| \ldots\underbrace{|j|i|}_{\pi_{\ell}}\ldots| \pi_{k} ] & \mtxt{ if the component string } \pi_{\ell} \mtxt{ is
  of the form }  ji.
\end{cases}
\end{equation}

We extend $\sd_{ij}$ for $i<j$ by declaring $\sd_{ij} \defeq \sd_{ji}$. In particular if $\alpha$ is the standard positive simple
root of $SL_{n}$ then $\sd_{\alpha}$ makes sense.
\end{defn}

The subdivision operator can change the composition type of a barred permutation. Moreover, it is fairly easy to see
that $\sd_{ij}$ is not $S_{n}$-equivariant with the action described in \cref{L:WGrpAct} i.e., $\sd_{ij} \circ \sigma \neq
\sigma \circ \sd_{ij}$ on $\mcal{B}_{n}$.

%  ]]]

%%% Local Variables:
%%% mode: latex
%%% TeX-master: "MC-0"
%%% End:

%                 +++ Ending file MC-0-sec3.tex +++   
%------------------------------------------------------------------------------------------------------------------------------------------------

%                 +++ Starting file MC-0-sec4a.tex +++   

% T STABLE CURVES AND STUFF [[[
\subsection{Towards a GKM theory of complete quadrics}\label{SS:coh}
GKM theory and its extension to algebraic varieties by Brion provides a powerful tool to calculate equivariant and (in
many cases) non-equivariant cohomology. To apply this theory in the context of complete quadrics one needs to answer
the following two questions.   

\begin{qstn}\label{Q:one}
    Given a codimension one algebraic subtorus $T^{\prime} \subset T$ classify the positive dimensional irreducible
    components $Y$ of the fixed point varieties $\mcal{X}_{n}^{T^{\prime}} \subset \mcal{X}_{n}$. 
\end{qstn}

An important feature of smooth, projective, spherical varieties is that such components $Y$ are either isomorphic to
$\mbb{P}^{1}$ or when $T^{\prime} = \Ker(\alpha)$, for some positive simple root $\alpha$, then $Y$ is either isomorphic
to $\mbb{P}^{1}$ or it is a $SL_{2}$-spherical variety isomorphic to $\mbb{P}^{2}$ or a rational ruled surface, see \cite{BanerjeeCan,Brion97} for
details. In the particular case of complete quadrics we can precisely work out the irreducible components of the
$T^{\prime}$-fixed subvarieties. 

\begin{notn}
We introduce some notation that will be used throughout the rest of this section. Consider a torus fixed point $\pi \in
\mcal{X}_{n}$. Let $\mu(\pi)$ denote the (special) composition indexing the $G$-orbit $G \cdot \pi $,
i.e. $\mcal{X}^{\mu(\pi)} = \overline{G \cdot \pi}$. Let $I(\pi) $ denote the subset of simple roots $\Delta$
corresponding to $\mu(\pi)$ and $p_{\pi}: \mcal{X}^{\mu(\pi)} \rarr G/P_{I(\pi)}$  denote the canonical
projection\footnote{Note: we deviate from denoting the projection as $\pi_{p}$ as in \cref{R:Fibstr} to avoid the awkward
  notation $\pi_{\pi}$.}. We denote the parabolic subgroup of the Weyl group $W$ by $W_{I(\pi)}$ and $W^{I(\pi)} \defeq
W/W_{I(\pi)}$. We call a torus fixed point $\pi$ special if $p_{\pi}(\pi)$ is the coset of the standard parabolic
subgroup $P_{I(\pi)}$ in
$G/P_{I(\pi)}$. 

Let $T^{\prime} = \Ker(\delta)$ denote a codimension one subtorus of $T$ for some root $\delta$. 
\end{notn}

It follows from \cite[\S 7]{DP83}  that the tangent space $T_{\pi}$ at $\pi$ in $\mcal{X}_{n}$ admits a $T$-stable
direct sum decomposition 
\begin{equation}\label{E:tangentSpace}
    T_{\pi}  =  T^{h}_{\pi} \oplus T^{v}_{\pi} \oplus T^{n}_{\pi},
\end{equation}
where
\begin{enumerate}
\item $T_{\pi}^{h}$ is isomorphic to the tangent space of $G/P_{I(\pi)}$ at the point $p_{\pi}(\pi)$;
\item $T_{\pi}^{v}$ is the tangent space of the  fiber  $p_{\pi}^{-1}(p_{\pi}(\pi))$;
\item $T^{n}_{\pi}$ is the stalk of the normal bundle to $\mcal{X}^{\mu(\pi)} \hookrightarrow \mcal{X}_{n}$ at the point
$\pi$. 
\end{enumerate}

The general idea is that when the point $\pi$ is special the summands in \cref{E:tangentSpace} can be explicitly computed in terms of certain subsets of the root system $\Phi$. When $\pi$ is not special we can always find a special point $\pi^{\prime}$ such that $\pi = w(\pi^\prime)$ for
some, possibly non-unique, $w \in W$, and at the level of tangent spaces we get $T_{\pi} = w(T_{\pi^{\prime}})$. All
such possible choices of $w \in W$ correspond to a unique element of $W^{I(\pi)}$. In this case a $T^{\prime} = \Ker(\delta)$ fixed
subspace at $T_{\pi}$ corresponds to a $\Ker(w^{-1}\cdot \delta)$ fixes subspace at $T_{\pi^{\prime}}$. 

When $\pi$ is special, $T_{\pi}^{h}$ is isomorphic to the Lie-algebra of the unipotent radical $\mfrk{U}(P_{I(\pi)}^{-})$ of the opposite parabolic
subgroup $P_{I(\pi)}^{-}$. Denoting the roots appearing in $\mfrk{U}(P_{I(\pi)}^{-})$ by $\Phi_{h}$, any
$T^{\prime}$-fixed subvariety has tangent space contained in $T_{\pi}^{h}$ if and only if $\pm\delta \in \Phi^{-}_{h}$ and all such
subvarieties are isomorphic to $\mbb{P}^{1}$ with $T$ acting by weight $\delta$. %%% WE are A_n root system so no more
                                %%% than two roots 

\begin{lema}\label{L:vertical-space}
Assume $\pi$ is special, then we have a $T$-equivariant decomposition
\[ T_{\pi}^{v} = \oplus_{\alpha \in I(\pi)} \mfrk{sl}_{2, \alpha}/\mfrk{so}_{2, \alpha} = \oplus_{\alpha \in I(\pi)} (k_{-\alpha} \oplus k_{\alpha}) \]
where $\mfrk{sl}_{2, \alpha}$ corresponds to the unique $\mfrk{sl}_{2}$  pair in the Lie algebra $\mfrk{sl}_{n}$
corresponding to the simple root $\alpha$.
\end{lema}
 Let $\Phi_{v}$ denote the set of negative roots $\{ \alpha: \alpha \in I(\pi) \}$. Then we have $T^{\prime}$-fixed subvariety if
 and only if $\delta = \pm \alpha$. In this case all such $T^{\prime}$-fixed subvarieties are isomorphic to
 $\mbb{P}^{2}$ and the maximal torus $T$ acts on a generic point with weight $\pm 2\alpha$.

\begin{lema}\label{L:normal-space}
    When $\pi$ is a special, we have $T$-weight space decomposition
    \[ T_{\pi}^{n} = \oplus_{\alpha \in \Delta \setminus I(\pi)}k_{-(\alpha + w_{I(\pi)}(\alpha))},\]
    where $w_{I(\pi)}$ is the longest element in the Weyl-group $W_{I(\pi)}$.
\end{lema}
Let $\Phi_{n}$ denote the set of roots $\{-(\alpha + w_{I(\pi)}(\alpha)): \alpha \in \Delta \setminus I(\pi) \}$. Then we have $T^{\prime}$-fixed subvariety if
and only if $\delta = \pm(\alpha + w_{I(\pi)}(\alpha))$. In this case all such $T^{\prime}$-fixed subvarieties are isomorphic to
 $\mbb{P}^{1}$ and the maximal torus $T$ acts on a generic point with weight $\delta$. 

\begin{rem}
In the case of quadrics, we refer the reader to~\cite[\S 2]{Strickland86} for detailed proofs of \cref{L:vertical-space} and \cref{L:normal-space}. Note that the set $I(\pi)$ in our notation corresponds to $J$ in \emph{loc. cit.} The general case
for any symmetric space is discussed in~\cite{DS85}.
\end{rem}

\begin{rem}\label{R:max-element}
We point out that by structure theory of $\msf{A}_{n}$-root systems, the simple reflections $s_{\beta}$ corresponding to $\beta \in I(\pi)$ commute. So
 $w_{I(\pi)}= \prod_{\beta \in I(\pi) } s_{\beta}$ where the product is taken in \emph{any} order. 
\end{rem}

\begin{rem} \label{R:str-fixed-pts}
The above analysis shows that the only two dimensional $T^{\prime}$ fixed varieties are along the fibers of the
projection map $p_{\pi}$ and it is isomorphic to $\mbb{P}^{2}$ viewed as an equivariant compactification of
$SL_{2}/SO_{2}$.
\end{rem}

The second question that one needs to answer is the following.
\begin{qstn}\label{Q:two}
   Suppose $Y \subset \mcal{X}_{n}^{T^{\prime}}$ is an any such component and $\pi$ is a given torus fixed point on $Y$. Then what are the other torus fixed points on $Y$?
\end{qstn}

Unfortunately, we do not have a complete satisfactory answer to this question. It is enough to consider the particular case when $\pi$ is
a special barred permutation. In this case, let us denote the irreducible component $Y$ by $Y_{\delta}$ where $\delta$
belongs to one of the subsets $\Phi^{\ast}$, where $\ast \in \{h,v,n \}$, as above. If $\delta$ is contained in $\Phi^{h}$ then $p_{\pi}(Y_{\delta})$
projects to a $T$-fixed curve in the flag variety $G/P_{I(\pi)}$ passing through $p_{\pi}(\pi)$. The structure of such
curves are known, see \cite[Lemma 2.2]{carr-kutt}, and it follows that the other torus fixed point is $\pi^{\prime} = r_{\delta}
\cdot \pi$ where $r_{\delta} \in W$ is the reflection  associated to $\delta$.

If $\delta \in \Phi_{v}$, then thanks to the product structure of the fibers of $p_{\pi}$, we can reduce to the case of
$SL_{2}/SO_{2}$ and show that the other two torus fixed points are given by $\sd_{\delta}(\pi)$ and
$r_{\delta}(\sd_{\delta}(\pi))$, where $\sd_{\delta}$ is the subdivision operation  and $r_{\delta}$ is the
reflection associated to $\delta$.

When $\delta \in \Phi_{n}$ we do not have a characterization of the other torus fixed point on $Y_{\delta}$.  

\begin{rem}
    We note that a complete description of the torus fixed points in the $T^{\prime}$-fixed curves $Y_{\delta}$
    corresponding $\pm \delta \in \Phi_{n}$ will immediately give us a presentation of the $T$-equivariant Chow cohomology ring, using the results of
    Brion in \cite[\S 7]{Brion97}. Combined with well known results about the isomorphism of cycle-class maps for smooth, projective varieties and
    equivariant formality of algebraic varieties we will get a new presentation of the cohomology ring of complete
    quadrics.
\end{rem}

%    ]]]

%%% Local Variables:
%%% mode: latex
%%% TeX-master: "MC-0"
%%% End:

%                 +++ Ending file MC-0-sec4a.tex +++   
%------------------------------------------------------------------------------------------------------------------------------------------------

%                 +++ Starting file MC-0-sec5.tex +++   

\section{A geometric order on degenerate involutions} \label{S:Bruhat} In this section we aim to study the Bruhat order
on all $B$-orbits or equivalently the order on all $\mu$-involutions as $\mu$-varies over all compositions of $n$. The
covering relations in this order come in two flavors: (a) covering relations between $\mu$-involutions for a fixed
composition, and (b) covering relations between involutions corresponding to different compositions.

There is a general recursive characterization of the Bruhat order on any spherical variety due to Timashev
\cite{Timashev94} using the action of the Richardson-Springer (RS) monoid. In the first case, this provides enough
information. In the second case we use $\WW$-sets (see \cite{Brion98}) to get sharper results.

\subsection{Geometric ordering on $\mu$-involutions: composition $\mu$ is fixed}

Let us fix a composition $\mu = (\mu_1, \dots, \mu_k)$ of $n$ and let $\pi$ be any $\mu$-involution. The $B$-orbit
(resp. its closure) indexed by $\pi$ is denoted by $\mscr{O}^{\pi}$ (resp., $\mscr{X}^{\pi}$). The Bruhat-order on
$\mu$-involutions is given by
\[\pi \leq \pi' \; \mtxt { if and only if  } \; \mscr{X}^{\pi} \subseteq \mscr{X}^{\pi'}.
\]

This is a ranked poset with unique maximum and minimum elements and rank function 
\[
\rank(\pi) := L_{\mu}(\min) - L_{\mu}(\pi).
\]
where $L_{\mu}(-)$ is the length function on a $\mu$-involution defined in \cref{eqn:mu-invlength}. Timashev's recursive
description on $\mu$-involutions (for a fixed composition) is as follows.

\begin{prop}\label{prop:recursive}
Let $\pi$ and $\rho$ be two $\mu$-involutions. Then $\pi \leq \rho$ in the Bruhat order if and only if
    \begin{enumerate}[label={(\roman*)}]
    \item $\pi = \rho$; or
    \item there exists $\mu$-involutions $\pi^* \leq \rho^*$ and a simple transposition $s_\alpha$
      such that under the RS-monoid actions we have $\rho = s_{\alpha} \cdot \rho^{\ast}, \pi = s_{\alpha} \cdot
      \pi^{\ast}$ and $\rho^{\ast} \neq \rho$. 
    \end{enumerate}
\end{prop}

The proposition follows from \S 2.9 of~\cite{Timashev94}. The covering relations have the following 
concrete description. Let $\pi, \rho$ be two $\mu$-involutions then $\pi \lessdot \rho$ is a covering relation if and
only if there exists a permutation $w \in S_n$, a simple transposition $s_\alpha \in S_{n}$ and two $\mu$-involutions $\pi^*, \rho^*$ satisfying all of the following conditions.
\begin{enumerate}[label={(\roman*)}]
\item $\pi = w \cdot \pi^*$;
\item $\rho = w \cdot \rho^*$;
\item compatibility with the length function $L_{\mu}$ and $\ell$:
  \[L_{\mu}(\pi) = L_{\mu}(\pi^*) + \ell(w), \; L_{\mu}(\rho) = L(\rho^*) + \ell(w) \; \text{ and } \]
\item a weak covering relation $\rho^* = s_{\alpha} \cdot \pi^*$.
\end{enumerate}

\begin{rem}\label{R:cex}
  This description above is concrete but it is not well suited for computations. As we observe below the poset structure
  for general compositions can be very different than special ones. 

  Consider two extreme compositions: $\mu = (1,1,\dots,1)$ and $\mu = (n)$. In the first case, $\mu$-involutions are the
  same as elements of $S_n$  and in the second case they are the involutions in $S_{n}$. The restriction of the Bruhat order on
  $S_{n}$  to involutions and the Bruhat order on involutions agree (see \cite{Incitti04}). 

  This fails for a general $\mu$-involution. A $\mu$-involution is easily identified with a permutation in $S_n$ -
  in one-line notation this is simply the concatenation of the underlying components of $\mu$. But the Bruhat order in
  the $\mu$-involutions differs from the restriction of the Bruhat order on $S_{n}$. The former must be graded,see
 \cite{RennerBook}, but as illustrated in Figure~\ref{fig:notgraded} in the case of $S_{4}$ and $\mu = (3,1)$, the
  latter is not always graded. Consider the interval from $[432|1]$ to $[321|4]$ in the bottom right portion of the figure.
\begin{figure}[htp]
\centering
\begin{tikzpicture}[thick,scale=0.35, every node/.style={scale=0.65}]
  \begin{scope}[shift={(9,0)}]
\node [shape=rectangle,draw,fill=white] at (0,0) (a) {$[432|1]$};
\node [shape=rectangle,draw,fill=white]  at (-5,5) (b1) {$[243|1]$};
\node [shape=rectangle,draw,fill=white]  at (0,5) (b2) {$[324|1]$};
\node [shape=rectangle,draw,fill=white]  at (5,5) (b3) {$[431|2]$};
\node [shape=rectangle,draw,fill=white]  at (-7.5,10) (c1) {$[143|2]$};
\node [shape=rectangle,draw,fill=white]  at (-2.5,10) (c2) {$[234|1]$};
\node [shape=rectangle,draw,fill=white] at (2.5,10) (c3) {$[314|2]$};
\node [shape=rectangle,draw,fill=white]  at (7.5,10) (c4) {$[421|3]$};
\node [shape=rectangle,draw,fill=white]  at (-7.5,15) (d1) {$[134|2]$};
\node [shape=rectangle,draw,fill=white]  at (-2.5,15) (d2) {$[142|3]$};
\node [shape=rectangle,draw,fill=white]  at (2.5,15) (d3) {$[214|3]$};
\node [shape=rectangle,draw,fill=white]  at (7.5,15) (d4) {$[321|4]$};
\node [shape=rectangle,draw,fill=white]  at (-5,20) (e1) {$[124|3]$};
\node [shape=rectangle,draw,fill=white]  at (0,20) (e2) {$[132|4]$};
\node [shape=rectangle,draw,fill=white]  at (5,20) (e3) {$[213|4]$};
\node [shape=rectangle,draw,fill=white]  at (0,25) (f) {$[123|4]$};
\draw[-,thick] (a) to (b1);
\draw[-,thick] (a) to (b2);
\draw[-,thick] (a) to (b3);
\draw[-,thick] (b1) to (c1);
\draw[-,thick] (b1) to (c2);
\draw[-,thick] (b1) to (c4);
\draw[-,thick] (b2) to (c2);
\draw[-,thick] (b2) to (c3);
\draw[-,thick] (b3) to (c1);
\draw[-,thick] (b3) to (c3);
\draw[-,thick] (b3) to (c4);
%\draw[-,thick] (b2) to (d4);
\draw[-,thick] (c1) to (d1);
\draw[-,thick] (c1) to (d2);
\draw[-,thick] (c2) to (d1);
\draw[-,thick] (c2) to (d3);
\draw[-,thick] (c2) to (d4);
\draw[-,thick] (c3) to (d1);
\draw[-,thick] (c3) to (d3);
\draw[-,thick] (c3) to (d4);
\draw[-,thick] (c4) to (d2);
\draw[-,thick] (c4) to (d3);
\draw[-,thick] (c4) to (d4);
\draw[-,thick] (d1) to (e1);
\draw[-,thick] (d1) to (e2);
\draw[-,thick] (d2) to (e1);
\draw[-,thick] (d2) to (e2);
\draw[-,thick] (d3) to (e1);
\draw[-,thick] (d3) to (e3);
\draw[-,thick] (d4) to (e2);
\draw[-,thick] (d4) to (e3);
\draw[-,thick] (e1) to (f);
\draw[-,thick] (e2) to (f);
\draw[-,thick] (e3) to (f);
\end{scope}
\qquad
\begin{scope}[shift={(-9,0)}]
\node [shape=rectangle,draw,fill=white] at (0,0) (a) {$[432|1]$};
\node [shape=rectangle,draw,fill=white]  at (-5,5) (b1) {$[243|1]$};
\node [shape=rectangle,draw,fill=white]  at (0,5) (b2) {$[324|1]$};
\node [shape=rectangle,draw,fill=white]  at (5,5) (b3) {$[431|2]$};
\node [shape=rectangle,draw,fill=white]  at (-7.5,10) (c1) {$[143|2]$};
\node [shape=rectangle,draw,fill=white]  at (-2.5,10) (c2) {$[234|1]$};
\node [shape=rectangle,draw,fill=white] at (2.5,10) (c3) {$[314|2]$};
\node [shape=rectangle,draw,fill=white]  at (7.5,10) (c4) {$[421|3]$};
\node [shape=rectangle,draw,fill=white]  at (-7.5,15) (d1) {$[134|2]$};
\node [shape=rectangle,draw,fill=white]  at (-2.5,15) (d2) {$[142|3]$};
\node [shape=rectangle,draw,fill=white]  at (2.5,15) (d3) {$[214|3]$};
\node [shape=rectangle,draw,fill=white]  at (7.5,15) (d4) {$[321|4]$};
\node [shape=rectangle,draw,fill=white]  at (-5,20) (e1) {$[124|3]$};
\node [shape=rectangle,draw,fill=white]  at (0,20) (e2) {$[132|4]$};
\node [shape=rectangle,draw,fill=white]  at (5,20) (e3) {$[213|4]$};
\node [shape=rectangle,draw,fill=white]  at (0,25) (f) {$[123|4]$};
\draw[-,thick] (a) to (b1);
\draw[-,thick] (a) to (b2);
\draw[-,thick] (a) to (b3);
\draw[-,thick] (b1) to (c1);
\draw[-,thick] (b1) to (c2);
\draw[-,thick] (b2) to (c2);
\draw[-,thick] (b2) to (c3);
\draw[-,thick] (b3) to (c1);
\draw[-,thick] (b3) to (c3);
\draw[-,thick] (b3) to (c4);
\draw[-,thick] (b2) to (d4);
\draw[-,thick] (c1) to (d1);
\draw[-,thick] (c1) to (d2);
\draw[-,thick] (c2) to (d1);
\draw[-,thick] (c2) to (d3);
\draw[-,thick] (c3) to (d1);
\draw[-,thick] (c3) to (d3);
\draw[-,thick] (c4) to (d2);
\draw[-,thick] (c4) to (d3);
\draw[-,thick] (c4) to (d4);
\draw[-,thick] (d1) to (e1);
\draw[-,thick] (d1) to (e2);
\draw[-,thick] (d2) to (e1);
\draw[-,thick] (d2) to (e2);
\draw[-,thick] (d3) to (e1);
\draw[-,thick] (d3) to (e3);
\draw[-,thick] (d4) to (e2);
\draw[-,thick] (d4) to (e3);
\draw[-,thick] (e1) to (f);
\draw[-,thick] (e2) to (f);
\draw[-,thick] (e3) to (f);
\end{scope}
\end{tikzpicture}
\caption{The left hand side depicts the poset of $(3,1)$-involutions with induced ordering from $S_{4}$. It is not
  graded. The right hand side depicts the geometric ordering on $(3,1)$-involutions. It is graded.}
\label{fig:notgraded}
\end{figure}
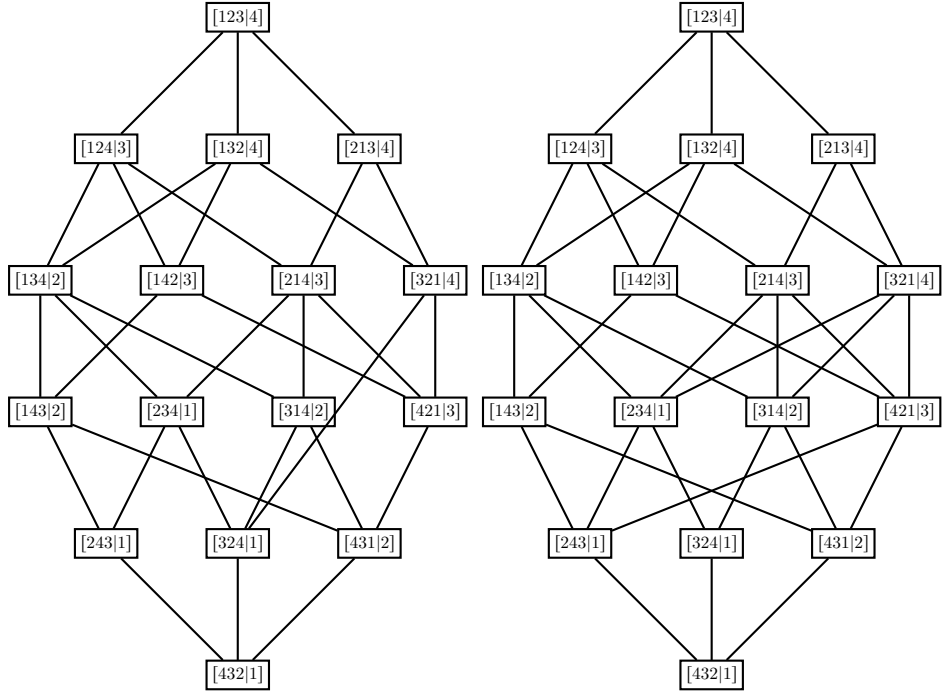

\end{rem}

\subsection{Geometric order on $\mu$-involutions for different compositions $\mu$}\label{SS:InclusiononallB}
We begin with some general remarks that apply for arbitrary connected, reductive algebraic group $G$ and a fixed Borel
subgroup $B$. We consider the order relations between two $B$-orbits contained in two different $G$-orbits. We begin by recalling
cancellative group actions on spherical varieties. 

\begin{defn}
The $G$-action on a spherical $G$-variety $X$ is called \emph{cancellative} if for any two distinct $B$-orbit closures
$Y_1$ and $Y_2$  in $X$, and for any  minimal parabolic subgroup $P_{\alpha}$, associated to a simple root $\alpha$ of $G$ (with respect to $B$), such that  
  $P_\alpha \cdot Y_1 \neq Y_1$, $P_\alpha Y_2 \neq Y_2$, we have
 \[ P_\alpha \cdot Y_1 \neq P_\alpha \cdot Y_2.\]
 \end{defn}

\begin{rem}
The $G$-action on flag varieties is cancellative, as is the $G\times G$-action on $G$.
However, the diagonal action of $G = SL_2$ on $\PP^1 \times \PP^1$ is not cancellative, see \cite{Brion98}.
\end{rem}

\begin{prop}\label{prop:cancellative}
      The $SL_{n}$-action on $\mcal{X}_{n}$ is cancellative. 
\end{prop}

\begin{proof}
  We set $G= SL_{n}$. The proof will use induction on $n$. The case of $n=1$ is clear because it is vacuously true. 

     In general, suppose  $Y$ is any $B$-orbit closure in $\mcal{X}_{n}$. Then there are two possibilities. 
     \begin{enumerate}[label={(\roman*)}]
      \item The intersection of $Y$ with the dense open $G$-orbit $\mcal{X}^{0}$ is nonempty. In this case, the intersection
            $Y \cap \mcal{X}^{0}$ is open dense and $B$-stable in $Y$.
          \item $Y$ is contained in the boundary $\mcal{X}_{n}\setminus \mcal{X}^{0}$.
          \end{enumerate}
          
   The group action commutes with taking closures. So to show that the action of $G$ on $\mcal{X}_{n}$ is cancellative it
     suffices to show that the action of $G$ on $\mcal{X}^{0}$ is cancellative, as well as the action of $G$ in each stratum $\mcal{X}^{\mu}$,
     where $\mu$ varies over compositions of $n$ with more than one part.

      In the first case, the weak order on the set of involutions is cancellative so the action of $G$ on $\mcal{X}^{0}$ is cancellative as well.  In the second case, we recall (see \cref{R:Fibstr}) that we have $G$-equivariant isomorphisms $\mcal{X}^{\mu} \cong G \times_{P_{\mu}} \mbb{X}^{\mu}$. Using Lemma
      1.2 of \cite{Brion98} it suffices to show that the the $L_{\mu}^{ss}$ action on $\mbb{X}^{\mu}$ is
      cancellative. But $L_{\mu}^{ss}$ is a product of $SL_{m}$ for $m \leq n$ and $\mbb{X}^{\mu}$ is a product of
      smaller rank symmetric spaces of same type. Direct product of cancellative action remains cancellative so the 
      proposition follows from the inductive hypothesis. 

\end{proof}

Cancellativeness is useful in the study of $\WW$-sets. In the context of complete quadrics, 
consider the intersection of of a $B$-stable subvariety $\mscr{X}^{\pi}$ and $G$-stable subvariety $\mcal{X}^{\mu}$,
where $\pi$ is a $\nu$-involution and $\mu$ is not necessarily not equal to $\nu$. Then the  decomposition of
$\mscr{X}^{\pi} \cap \mcal{X}^{\mu}$ into irreducible components is given by 
intersection \begin{equation}\label{E:brion-str-res}
  \mscr{X}^{\pi} \cap \mcal{X}^{\mu} = \bigcup\limits_{\gamma} \mscr{X}^{\gamma},
\end{equation}
where $\gamma$ runs over the set of all $\mu$-involutions such that the $\WW$-sets $\WW(\rho) \subset \WW(\pi)$ and
$L_{\mu}(\gamma) = L_{\nu}(\gamma)$; see \cite[Theorem 1.4]{Brion98}.   

\begin{lema}\label{T:W-setcriterion}
Let $\pi$ be a $\mu$-involution and $\rho$ be a $\nu$-involution and assume $\nu \precneq \mu$.
Then $\rho \leq \pi$ if and only if there exists a $\nu$-involution $\gamma$ with
$W(\gamma) \subseteq W(\pi)$ and
$\rho \leq \gamma$.
\end{lema}

\begin{proof}
      It is clear that if $\WW(\gamma) \subset \WW(\pi)$, for some $\nu$-involution $\gamma$, then $\mscr{X}^{\gamma}$ is contained in
      the intersection $\mscr{X}^{\pi} \cap \mcal{X}^{\nu}$, and hence in  $\mscr{X}^{\gamma} \subset \mscr{X}^{\pi}$). Moreover $\rho \leq \gamma$ so $\mscr{X}^{\rho} \subset \mscr{X}^{\gamma}$. This is proves the sufficiency. 

      On the other hand, from \cref{E:brion-str-res}, the intersection $\mscr{X}^{\pi} \cap \mcal{X}^{\mu}$ is a
      union of $\mscr{X}^{\gamma}$ such that $W(\gamma) \subset W(\pi)$. So if $\rho \leq \pi$, i.e. $\mscr{X}^{\rho} \subset \mscr{X}^{\pi}$, then clearly $\rho \subset \gamma$,
      for some $\nu$-involution $\gamma$ and $\WW(\rho) \subset \WW(\pi)$.  
\end{proof}

In Timashev's characterization, one starts with weak order covering relations in a
fixed $G$-orbit and then builds covering relations going `upward'. We construct a new order, based on the same principle, but going in the opposite direction i.e., starting from an opposite weak covering and moving `downward'
recursively.

More precisely consider the relation $\lessdot_{r}$ on the set of $\mu$-involutions defined below. Let $\star$ denote the
opposite action of the RS-monoid on $\mu$-involutions, i.e. $s \star \rho = \pi$ if and only if $s \cdot \pi = \rho$
(see \cref{eqn:RSaction}). Note that this is well defined because the original action of the $RS$-monoid is cancellative.

\begin{defn}\label{D:recursive}
If $\pi$ and $\rho$ are two $\mu$-involutions. Let $\leq$ (without the subscript $r$) denote the Bruhat order.  We
define $\rho \lessdot_{r} \pi$ 
if and only if either $\pi \lessdot_{W} \rho$ in the weak order, or there exist $\mu$-involutions $\pi^*$,
$\rho^*$, $w$ in the RS-monoid, and simple transposition $s$ such that
\begin{itemize}
    \item $\rho = w \star \rho^*$ with $L_\mu(\rho^*) = L_\mu(\rho)+\ell(w)$;
    \item $\pi = w \star \pi^*$ and $L_\mu(\pi^*)=L_\mu(\pi)+\ell(w)$;
\item $\rho^* \lessdot_{r} \pi^{*}$ and $\pi^{*} = s \cdot \rho^{*}$.
    \end{itemize}
\end{defn}

Let us denote the partial order $\leq_{r}$ on the set of $\mu$-involutions such with transitive closure of $\lessdot_{r}$ above. The partial order $\leq_{r}$ is compatible with the $\star$ action
of the RS-monoid (see Definition 5.3 \cite{RS90}). We call the partial order $\leq_{r}$ the \emph{reverse Bruhat order}. It is not clear apriori that the reverse Bruhat order is equal to the opposite Bruhat order and this will be established below. 

\begin{thm}\label{T:reverseBruhat}
      The reverse Bruhat order on $\mu$-involutions is the equal to the opposite of the Bruhat order. 
\end{thm}

\begin{proof}
  We will show that given two $\mu$-involutions $\rho$ and $\pi$ if $\pi \lessdot \rho$ (i.e. $\rho$ covers $\pi$ in the
  Bruhat order) then $\rho \lessdot_{r} \pi$ is a covering relation in the reverse Bruhat order and
  vice-versa. Throughout this proof We will use $\max$ (resp. $\min$) as the maximum and minimum element of the
  $\mu$-involutions in the Bruhat order. The strategy of proof is to systematically apply Timashev's recursive
  characterization while keeping the while moving along chains; we believe the depiction in Figure~\ref{F:Depiction}
  will aid the reader through the proof.
  
  Concretely, We want to show that if $\pi \lessdot \rho$ then $\pi \lessdot_{r} \rho$.
  Suppose $\pi= \max$. Then $\rho$ must also be $\max$ and there is nothing to prove. So, inductively we may
  assume that the hypothesis is true for all $\mu$-involutions with rank $ \geq \aleph + 1$ ( with respect to the rank
  function (\ref{eqn:mu-invlength}). We consider the
  case when $L_{\mu}(\pi) = \aleph$.

We define recursively $\mu$-involutions $\pi_{k}^{\ast}$, $\rho_{k}^{\ast}$, and simple reflections $s_{i_{k+1}}$ in $\mcal{M}(S_{n})$
(the RS-monoid). 
\begin{enumerate}[label={(\roman*)}]

\item We set $\pi_{0}^{\ast} = \pi$, $\rho_{0}^{\ast} = \rho$. Clearly $ \pi_{0}^{\ast} \lessdot \rho_{0}^{\ast}$ and if
  moreover it is a weak order cover then we must have for some simple reflection $s_{0}$ such that
  $s_{0} \cdot \pi_{0}^{\ast} = \rho_{0}^{\ast}$ and we set $s_{i_{1}} = 1$ and terminate.
  
\item At stage $k \geq 0$, $\pi_{k}^{\ast}$ and $\rho_{k}^{\ast}$ are given such that $\pi_{k}^{\ast} \lessdot
  \rho_{k}^{\ast}$ and if moreover there is a weak order cover then we must have simple reflection $s_{k}$ such that $s_{k} \cdot \pi_{k}^{\ast} = \rho_{k}^{\ast} $. We set $s_{i_{k+1}} = 1$ and  and terminate.
  
\item Otherwise we let $s_{i_{k+1}}$ denote a simple reflection such that $\pi^{\ast}_{k+1} = s_{i_{k+1}}\cdot \pi^*_k$
  and $\pi^*_{k+1} \neq \pi_k$; \; $\rho_{k+1}^{\ast} = s_{i_{k+1}} \cdot \rho^*_k$ and $\rho^*_{k+1}\neq
  \rho^{\ast}_{k}$ and repeat the previous step. The existence of $s_{i_{k+1}}$ is guaranteed by Timashev's characterization. 

\end{enumerate}

Let us suppose that, for the given $\rho$ and $\pi$ the algorithm terminates after $m$-steps i.e., we have a weak order covering relation
 $\pi_{m}^{\ast} \lessdot \rho_{m}^{\ast}$ (in the Bruhat order) and a simple reflection $s_{m}$ such that $s_{m} \cdot
 \pi_{m}^{\ast} = \rho_{m}^{\ast}$.  

 Let $\varpi$ be an element of the $\WW$-set of $\WW(\pi)$ and we fix a reduced expression $ \varpi = s_{j_1} s_{j_2}\cdots
 s_{j_k}$. The co-rank of $\pi$ must be $k$. 

Consider the element in the RS-monoid given by
\begin{equation} \label{E:omit-one}
w \defeq s_{j_1} \cdots s_{j_k} \cdot s_{i_m} \cdots s_{i_1}.
\end{equation}
The element $w$, by construction, is a member of the $\WW$-set
$\WW(\pi_{m}^{\ast})$ and we also have $s_m \cdot \pi_{m}^{\ast} = \rho_m^*$. 

In the opposite Bruhat order the action of the RS-monoid gives us the following equations.
\begin{align*}
  s_{m} \star \rho^{*}_{m} &= \pi^{*}_{m} & \mtxt{ (by definition)} \\
  w \star \max &= \pi_{m}^{*}  & \mtxt{ (by definition).}
\end{align*}

We apply the exchange property, \cite[PROPERTY 5.12(e)]{RS90}, to the opposite Bruhat order. In the notation of \emph{loc. cit.} we let
$x=\pi_m^*$, $y=\rho_m^*$, $s=s_m$ and we conclude that there is an element $w^{\prime}$ in the RS-monoid with the following properties.
\begin{itemize}
\item $w^{\prime} \star \max = \rho_{m}^{\ast}$;
  %% We are working in the opposite order
  \item $w^{\prime} = s_{j_1} \cdots \widehat{s_{\alpha}} \cdots s_{j_{k}} s_{i_m} \cdots s_{i_1} \cdots s_{m}$ where
    $\widehat{s_{\alpha}}$ is a deleted simple reflection from the expression in \cref{E:omit-one}. 
(In other words we have established that in Figure~\ref{F:Depiction} that $\widetilde{s_{j_l}}=s_{j_l}$ for at-most one $j_l\neq \alpha$.)
  \end{itemize}
  
Note that $s_\alpha$ cannot belong to the set $\{ s_{i_1},\dots, s_{i_m}\}$ because we have the constraint $|L_{\mu}(\rho) -
L_{\mu}(\rho_{m}^{*})| = m
$.
So $s_\alpha$ belongs to the set $\{ s_{j_1},\dots, s_{j_k}\}$. In this case we have new elements $\rho^{\prime}$ and
$\pi^{\prime}$, depending on $\alpha$, defined below (also see Figure \ref{F:Depiction}).

\begin{enumerate}[label={(\alph*)}]
\item \label{I:one} If $\alpha = j_{1}$  then $\pi= \pi^{\prime}$ and $\rho = \rho^{\prime}$ and $\pi^{\prime}
  \lessdot_{r} \rho^{\prime}$ as intended.
  
\item \label{I:three} If $\alpha \neq j_{1}$, then  set $\rho^{\prime} = s_{j_{d-1}}\cdot \ldots \cdot s_{j_{1}} \cdot
  \rho$  and $\pi^{\prime} = s_{j_{d-1}} \cdots s_{j_{1}} \cdot \pi$. The induction hypothesis on rank,
  $L_{\mu}(\pi^{\prime}) > L_{\mu}(\pi)$ implies that $\pi^{\prime} \lessdot_{r} \rho^{\prime}$ and thus $\pi
  \lessdot_{r} \rho$. 
\end{enumerate}

The argument we have used is reversible. This is because they only depend on the abstract properties of RS-monoids and
compatibility of the orderings with the monoid action. So we can repeat it verbatim to show that if $\rho \lessdot_{r} \pi$ then $\pi \lessdot
\rho$.  This proves the result. 
\end{proof}

\begin{figure}[htp]
\centering
\begin{minipage}[c]{1.75 \textwidth}
\begin{tikzpicture}[scale=1.2, every node/.style ={scale=0.6}]

\begin{scope} [rotate = 90]

\node at (-.25,-.25) {$\pi_m^*$};
\node at ( 0 .65, 0.5) {$s_m$};
\node at (1.25,1) {$\rho_m^*$};
\node at (-.25,0.35) {$s_{i_m}$};
\node at (1.25,1.5) {$s_{i_m}$};
\node at (0,0) (a1) {$\bullet$};
\node at (0,1) (a2) {$\bullet$};
\node at (1,1) (b1) {$\bullet$};
\node at (1,2) (b2) {$\bullet$};
\node at (-.75,2) (a3) {$\bullet$};
\node at (0.25,3) (b3) {$\bullet$};
\node at (-.65,1.45) {$s_{i_{m-1}}$};
\node at (.85, 2.56) {$s_{i_{m-1}}$};

\draw[-, thick] (0,0) to (0,1);
\draw[-, thick] (0,0) to (1,1);
\draw[dashed, thick] (0,1) to (1,2);
\draw[dashed, thick] (-.75,2) to (.25,3);
\draw[-,thick] (1,1) to (1,2);
\draw[-, thick] (0,1) to (-.75,2);
\draw[-, thick] (1,2) to (0.25,3);
\draw[dotted, thick] (-0.75,2) to (-0.75,4);
\draw[dotted, thick] (0.25,3) to (.25,5);
\draw[dashed,thick] (0,1) to (1,2);

\node at (-0.75,4) (a5) {$\bullet$};
\node at (.25,5) (b5) {$\bullet$};
\node at (-1,3.75) {$\pi_1^*$};
\node at (.5,5) {$\rho_1^*$};
\node at (-.9,4.4) {$s_{i_{1}}$};
\node at (.5,5.5) {$s_{i_{1}}$};

\draw[dashed,thick] (-.75,4) to (.25,5);
\draw[dashed,thick,red] (-.75,5) to (.25,6);

\node at (-0.75,5) (a6) {$\bullet$};
\node at (.25,6) (b6) {$\bullet$};
\node at (-1,4.75) {$\pi$};
\node at (.5,6) {$\rho$};
\draw[-,thick] (-0.75,4) to (-0.75,5);
\draw[-,thick] (.25,5) to (.25,6);
\node at (-1.2, 5.2) {$s_{j_1}$};
\node at (0,6.7) {$\widetilde{s_{j_1}}$};

\draw[-,thick] (-0.75,5) to (-1.5,6);
\draw[-,thick] (.25,6) to (-.5,7);
\node at (-1.5, 6) {$\bullet$};
\node at (-.5, 7) {$\bullet$};
\draw[dotted, thick] (-1.5,6) to (-1.5,8);
\node at (-.3, 7.6) {$\rho^{\prime}$};
\draw [dashed, thick] (-0.5, 7.5) to (-1.5, 6.5);
\node at (-1.5, 6.5){$\bullet$};
\node at (-1.75, 6.5) {$\pi^{\prime}$};
\draw[dotted, thick] (-.5,7) to (-.5,9);

\node at (-.5, 7.5){$\bullet$};

\node at (-1.5,8) {$\bullet$};
\node at (-.5,9) {$\bullet$};
\draw[-, thick] (-1.5,8) to (-1.5,9);
\node at (-1.5,9) {$\bullet$};
\node at (-1.75,8.5) {$s_{j_{k-1}}$};
\draw[-, thick] (-1.5,9) to (.5,10);
\draw[-, thick] (-.5,9) to (.5,10);
\node at (.5,10) {$\bullet$};
\node at (-.6,9.7) {$s_{j_{k}}$};
\node at (.25,9.4) {$\widetilde{s_{j_{k}}}$};
\node at (.5,10.25) {$\max$};

\node at (0.7, 7.5) {{\small Bruhat order} $\rightsquigarrow$ };

\node at (-2.3, 5.5) { $\leftarrow$ {\small reverse Bruhat order}};
\end{scope}
\end{tikzpicture}
\end{minipage}

\caption{The initial covering relation is depicted in red. The solid lines (and dotted lines) indicate weak order cover relations. The
  dashed lines are Bruhat order covers but they are not necessarily covers of the weak order.}

\label{F:Depiction}
\end{figure}

% \begin{rem}
%   We note that in the light of the previous proposition, for any $\mu$-involution $\pi$, the reverse $\WW$-set
%   $\WW^{-1}(\pi)$ with respect to the Bruhat order (see Definition \ref{def:w-set}) is same as the $\WW$-set of $\pi$ with
%   respect to the reverse Bruhat order.
% \end{rem}

\begin{cor}\label{lem:w1}
Let $\pi$ and $\rho$ be two $\mu$-involutions. Then $\pi$ is covered by $\rho$ in the Bruhat order if and only if given any element
$\varpi \in W^{-1}(\pi)$ then we can find a simple reflection $s_{\alpha}$ and elements $w_1,w_2\in W$ with the 
following property.

\begin{itemize}

\item Factorization: $ \varpi = w_1\cdot w_2$ and $w_1\cdot s_{\alpha}\cdot w_2\in W^{-1}(\rho)$; and

\item length constraint:
  \begin{equation}\label{E:lnth-const}
\ell(w_1\cdot w_2) = \ell(w_1) + \ell(w_2).
\end{equation}
 \end{itemize}
\end{cor}

\begin{proof}
Let $\min$ (resp. $\max$) denote the minimum
  (resp. maximum) elements of $\mu$-involutions with the Bruhat order.
  
Given a covering relation $\pi \lessdot \rho$ in the Bruhat order, we have $\rho \lessdot_{r}
\pi$ in the reverse. Applying the description of the covering relations in the reverse Bruhat order (see Definition \ref{D:recursive}) we obtain
the following (see Figure \ref{F:Depiction-2} for an illustration).
\begin{enumerate}[label={(\roman*)}]
\item An element $w_{2}\in W$ such that $\rho_{m} \defeq w_{2} \star \rho$, $\pi_{m} \defeq w_{2} \star \pi$, and a
  simple reflection $s_{\alpha}$ such that $s_{\alpha} \star \rho_{m} = \pi_{m}$ (i.e. the covering relation $\rho_{m}
  \lessdot_{r} \pi_{m}$ is a weak covering relation). 

\item Let $w_{1} \in W$ be any element of the $\WW$-set $\WW(\pi_{m})$ in the reverse order. In other words we get
  $w_{1} w_{2} \star \pi = \min$ and $w_{1}s_{\alpha}\cdot w_{2} \star \rho = \min$.
\end{enumerate}
The length constraint is a simple consequence of the properties of the rank (resp. co-rank) function $L_{\mu}$ of the
reverse Bruhat order (resp., Bruhat order).

Conversely, suppose we have elements $w_{1}, w_{2}, s_{\alpha}$ as asserted. We set $\pi_{m}^{\ast} = w_{2} \cdot \min$
and $\rho_{m}^{\ast} = s_{\alpha}\cdot w_{2} \cdot \min$.  Then clearly $\pi_{m}^{\ast} \leq \rho_{m}^{\ast}$ in the
weak order. The recursive definition of the weak order then implies $\pi \leq \rho$. 
\end{proof}

\begin{figure}[htp]
\centering
\begin{minipage}[c]{\textwidth}
\begin{tikzpicture}[scale=1.75, every node/.style={scale=0.9}]
\begin{scope}[rotate = 90]
\draw[-,thick] (0,0) to (0,-0.5);
\node at (0, -0.5) {$\bullet$};
\draw[dotted,thick] (0,-0.5) to (0,-1.5);
\node at (0, -1.5) {$\bullet$};
\draw[-,thick] (0,-1.5) to (0,-2);
\node at (0,-2) {$\bullet$};
\node at (0.5,-1.9) {$\min$};
\node at (-.31,0) {$\pi_m$};
\node at (1.31,1) {$\rho_m$};
\node at (.65,.35) {$s_{m}$};
\node at (0.2 , .55) {$s_{i_m}$};
\node at (1.25,1.5) {$s_{i_m}$};
\node at (0,0) (a1) {$\bullet$};
\node at (0,1) (a2) {$\bullet$};
\node at (1,1) (b1) {$\bullet$};
\node at (1,2) (b2) {$\bullet$};
\node at (-.75,2) (a3) {$\bullet$};
\node at (0.25,3) (b3) {$\bullet$};
\node at (-.6,1.1) {$s_{i_{m-1}}$};
\node at (1.1,2.5) {$s_{i_{m-1}}$};

\draw[-, thick] (0,0) to (0,1);
\draw[-, thick] (0,0) to (1,1);
\draw[dashed, thick] (0,1) to (1,2);
\draw[dashed, thick] (-.75,2) to (.25,3);
\draw[-, thick] (1,1) to (1,2);
\draw[-, thick] (0,1) to (-.75,2);
\draw[-, thick] (1,2) to (0.25,3);
\draw[dotted, thick] (-0.75,2) to (-0.75,3);
\draw[dotted, thick] (0.25,3) to (.25,4);
\draw[dashed, thick] (0,1) to (1,2);

\node at (-0.75,3) (a5) {$\bullet$};
\node at (.25,4) (b5) {$\bullet$};
\node at (-1.1,2.75) {$\pi_1$};
\node at (.6,4) {$\rho_1$};
\node at (-.9,3.4) {$s_{i_{1}}$};
\node at (.5,4.5) {$s_{i_{1}}$};
\draw[dashed, thick] (-.75,3) to (.25,4);
\draw[dashed, thick] (-.75,4) to (.25,5);
\node at (-0.75,4) (a6) {$\bullet$};
\node at (.25,5) (b6) {$\bullet$};
\node at (-1,4.0) {$\pi$};
\node at (.5,5) {$\rho$};
\draw[-, thick] (-0.75,3) to (-0.75,4);
\draw[-, thick] (.25,4) to (.25,5);
\node at (1.25,4.5) {{\small Bruhat order $\rightsquigarrow$ }};
\node at (-1.0 ,  -0.5 ) {{\small $\leftarrow$ reverse Bruhat order}};
\end{scope}
\end{tikzpicture}
\end{minipage}
\caption{The solid and dotted arrows represent weak order covers the dashed arrows represent covering relations.}
\label{F:Depiction-2}
\end{figure}
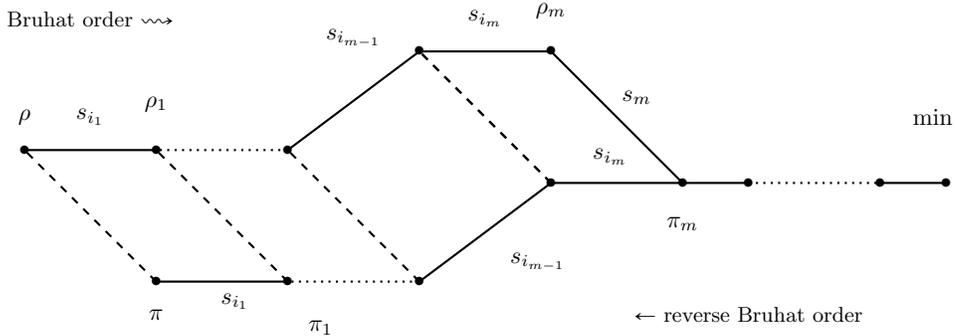

Now we present a complete description of the covering relations in
Bruhat order.
\begin{thm}\label{thm:Bruhat}
Let $\pi$ be a $\mu$-involution and $\rho$ be a $\nu$-involution. Then $\rho$ covers $\pi$ in Bruhat order if and only if one of the following holds:
\begin{enumerate}[label={(\roman*)}]
\item $\mu$ is covered by $\nu$ in the refinement ordering (see Definition \ref{E:refinement-order}) and
  $\WW(\pi) \subset \WW(\rho)$.

\item \label{I:T-2} The compositions $\nu = \mu$. Moreover, there exist a simple reflection $s_{\alpha}$  and an element $\varpi \in W$ such that
  \begin{enumerate}
  \item \label{I:T-21}
    $L_{\mu}(\pi) - L_{\mu}(\varpi \star \pi) = L_{\mu}(\rho) - L_{\mu}( \varpi \star \rho)=\ell(\varpi)$;
  \item \label{I:T-22} $s_{\alpha} \cdot (\varpi \cdot \pi) = \varpi \cdot \rho $ (equivalently in the reverse Bruhat
    order  $s_{\alpha} \star (\varpi \star \rho) = \varpi \star \pi$);
  \item \label{I:T-23}$ s_{\alpha}\WW^{-1}(\varpi \star \pi) \cap \WW^{-1}(\varpi \star \rho) \neq \emptyset$ where
    $s_{\alpha}\WW^{-1}(\varpi \star \pi)$ is the translation by group action i.e.  \[\{ s_{\alpha} w \in W \; | \; w \in
    \WW^{-1}(\varpi \star \pi) \}.\]
  \end{enumerate}
\end{enumerate}
\end{thm}

\begin{proof}
The first characterization follows easily from the description of $\WW$-sets and Proposition~\ref{T:W-setcriterion}.

We prove the second characterization. Since $\pi \lessdot \rho$ we have $\rho \lessdot_{r} \pi$. We set $\varpi =
w_{1}$ where $w_{1}, w_{2}, s_{\alpha}$ exist from \cref{lem:w1}. It follows
that Condition (\ref{I:T-21}) is clear. Condition (\ref{I:T-22}) follows because
\[\ell(w_{1}s_{\alpha}w_{2}) \leq \ell(w_{1})+ 1+ \ell(w_{2}) = L_{\mu}(\pi) + 1 = L(\rho),\] so $s_{\alpha} \in
\WW(\varpi \cdot \rho, \varpi \cdot \pi)$. Finally Condition (\ref{I:T-23}) follows because $s_{\alpha}w_{2} \in
s_{\alpha}\WW^{-1}(\varpi \star \rho) \cap \WW^{-1}(\varpi \star \pi) $.

Conversely, suppose we are given elements $\varpi$, $s_{\alpha}$ as above. Then we set $w_{1} = \varpi$ and choose any
element $w_{2}\in \WW^{-1}(\varpi \star \pi) \cap s_{\alpha}\WW^{-1}(\varpi \star \rho)$. Then $w_{2} \in \WW^{-1}(\varpi
\star \pi)$ and so  $w_{1}w_{2} \in \WW^{-1}(\pi)$. Similarly $w_{1}s_{\alpha}w_{2} \in \WW^{-1}(\rho)$. The length constraint
follows from the Condition (\ref{I:T-21}) above.  This proves the proposition. 
\end{proof}

\begin{exmple}
Consider $\pi=[21|3]$. It is easily computed that $\WW^{-1}([21|3])=\{312\}$.
$\pi$ covers two $(2,1)$-involutions $\rho_1=[31|2]$ and $\rho_2=[23|1]$.
Note that $\WW^{-1}(\rho_1)=\{ 213 \}$,  $\WW^{-1}(\rho_2)=\{ 132 \}$.
For the covering $\rho_1 \lessdot \pi$, the transposition is $s= 132$ and
for the covering $\rho_2 \lessdot \pi$, the transposition is $s=321$.
Finally, the $\WW$-set of $\pi$ is $\{213\}$, and the only composition that is finer
than $2+1$ is $1+1+1$. Among all $(1,1,1)$-involutions, the only degenerate involution
whose $\WW$-set is a subset of $\{213\}$ is $\rho_3=[2|1|3]$. Therefore, we found all
degenerate involutions that are covered by $\pi$; these are $\rho_1,\rho_2$
and $\rho_3$.
\end{exmple}

%%% Local Variables:
%%% mode: latex
%%% TeX-master: "MC-0"
%%% End:

%                 +++ Ending file MC-0-sec5.tex +++   
%------------------------------------------------------------------------------------------------------------------------------------------------

%                 +++ Starting file MC-0-sec6.tex +++   

\section{Cell decomposition and barred permutations} \label{S:cells-and-bars}
In the case of $\mcal{X}_{n}$ (more generally for any spherical variety) the finitely many Borel orbits provide a
stratification of $\mcal{X}_{n}$ with each stratum indexed by degenerate involutions. The BB-decomposition produces cells that also provide a geometrically meaningful paving of
$\mcal{X}_{n}$. Unfortunately, the BB-decomposition is rarely a stratification (i.e., the closure of a BB-cell is not
necessarily the union of other BB-cells). Nonetheless, one can still define a partial order on the $T$-fixed points by closure relations. Given barred permutations $\pi, \pi^{\prime}$ we define the BB-ordering  
\begin{equation} \label{E:bborder}
  \pi \leq \pi^{\prime} \Longleftrightarrow X_{\pi}^{+} \subset \overline{X_{\pi^{\prime}}^{+}},
\end{equation}
where $X_{\pi}^{+}$ (resp. $X_{\pi^{\prime}}^{+}$) are the corresponding BB-cells. In particular cases, the order complex of such
orderings have been investigated by Knutson, see \cite{Knutson}.

However, a Borel orbit is contained in a unique BB-cell and as a result, in the spherical case, there is a
maximal dimensional $B$-orbit which is contained in a given BB-cell of $\mcal{X}_{n}$.  This information will be encoded by two
combinatorial maps 
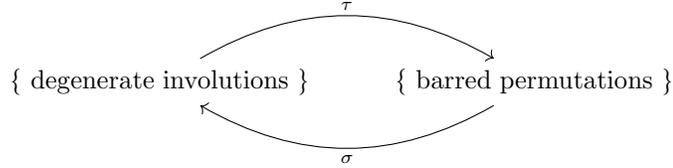
\begin{equation*}
  \begin{tikzcd}
   \{ \mtxt{ degenerate involutions } \} \arrow[r, bend left, "\tau"] & \{ \mtxt{ barred permutations } \}  \arrow[l, bend left, "\sigma"] 
  \end{tikzcd}
\end{equation*}
that we describe below.

\subsubsection{The map $\tau$}

\begin{defn}\label{D:taumap}
Consider the following function \[
\tau : \{ \text{degenerate involutions} \} \rightarrow \{ \text{barred permutations}\}.
\]

\begin{itemize}

\item Let $\pi = [\pi_1 | \pi_2 | \dots | \pi_k]$ be a degenrate involution. For each $\pi_j$, order its
  cycles in lexicographic order by the smallest value in each cycle.

  \item Since $\pi$ is a $\mu$-involution, every cycle that occurs in each $\pi_j$
    has length one or two. Then add bars between each cycle.

\item  Take the equivalence class of the resulting $\mu$-involution. Concretely, we will remove braces from one-cycles
  $(i)$ and for a two-cycle $(ij)$ with $i < j$ will be converted into a string $ji$.
\end{itemize}
\end{defn}

It is easy to see that the process is well defined. For example,
$\tau( (68) | (25)(4)(9) | (13)(7) ) = [86 | 4 | 52 | 9 | 31 | 7]$.

We will now show how this map is connected to the BB-decomposition under the action of a generic one-parameter subgroup.  
\begin{defn} \label{D:admissibility}

A sequence $(a_1, a_2, \dots, a_{n})$ of integers satisfying the following conditions is called an \emph{admissible}
sequence of length $n$. 
\begin{enumerate}
\item \label{I:ad-cond1} $\sum_{i=1}^{n}a_{i} = 0 $;
\item \label{I:ad-cond2} the sequence $a_{i}$ is monotonically increasing;
  \item \label{I:ad-cond3} given indices $ i \leq j \leq k \leq l$ we have \[ a_{j}-a_{i} \leq a_{l} - a_{k}.\]
  \end{enumerate}
  
\end{defn}

\begin{lema}\label{L:admissible}
  Equivalently admissible sequences have the following properties:
  \begin{enumerate}[label={(\alph*)}]
\item $\sum_{i=1}^{n} a_{i} = 0 $;
  \item $a_1 < a_2 < \cdots < a_n$;
\item if $i, j, k < l$, then $a_i + a_j < a_k + a_l$;
\item if $i, j < k$, then $2a_i < a_j + a_k$.
\end{enumerate}

Admissible sequences exist. In particular,  for any positive integer $n$ the sequence  $\left ( n +2^{i}-2^{n} \right )$ for $0 \leq i \leq n-1$ is an admissible sequence.
\end{lema}
\begin{proof}
  The first assertion is trivial. The second assertion follows from the following observations.
  \begin{enumerate}[label={(\roman*)}]
  \item The sequence $(2^{i}-1)$ satisfies all conditions of admissibility, in
    \cref{D:admissibility}, except the condition (\ref{I:ad-cond1}). This follows
    from binary expansion of these integers.   
\item The conditions (\ref{I:ad-cond1}) and (\ref{I:ad-cond2}) are translation invariant. In other words, if any finite
  ordered sequence of integers $(a_{i})$ satisfies these conditions then for any integer $b$ the sequence $(a_{i}+ b)$ also
  satisfies these conditions. 
   \end{enumerate}

 \end{proof}

 An admissible sequence $\mathbf{a} \defeq (a_{1}, \dots, a_{n})$ defines an \emph{admissible} one-parameter subgroup ($1$-psg),
denoted by $\lambda_{\mbf{a}}: \mbb{G}_{m} \rarr SL_{n}$.
\begin{prop}\label{prop:quadric-limit}
Given any $\mu$-involution $\pi$, and any admissible $1$-psg $\lambda$, the point $Q_{\pi} \in \mcal{X}_{n}$ flows to
the torus fixed point $Q_{\tau(\pi)}$. In other words
\[ \lim_{t \rightarrow 0} \lambda(t) \cdot Q_{\pi} = Q_{\tau(\pi)}.\]
\end{prop}

\begin{proof}
Let $\{\epsilon_{i}: i = 1, \ldots n \}$ denote the fundamental weights of $SL_{n}$. Given any admissible $1$-psg
$\lambda_{\mbf{a}}$ we scale the bilinear pairing between roots and co-roots such that $\langle \lambda_{\mathbf{a}}, \epsilon_{i} \rangle = a_{i}$.

It suffices to understand the flow under an admissible $1$-psg inside a closed $G$-orbit $\mcal{X}^{\pi}$. Consider the
barred permutations corresponding to $\pi = [\pi_{1}| \ldots| \pi_{k}]$, and the limit point $\tau(\pi) = [\pi_{1}^{\prime}|
\ldots| \pi_{m}^{\prime}]$. Using the description of the structure of the orbits $\mcal{X}^{\pi}$, see Remark
\ref{R:Fibstr}, it is clear that under the flow of a generic one parameter subgroup each component
$\pi_{i}$ of $\pi$ will independently fragment into sub-components $[\pi^{\prime}_{i_{1}}|\ldots|\pi^{\prime}_{i_{j}}]$ and
$\tau(\pi)$ will be obtained by concatenating these sub-components. In other words the barred permutation $\tau(\pi)$
will be obtained by successively subdividing $\pi$ and it suffices to show that the process outlined in \cref{D:taumap}
is indeed the right one. 

The allows us to reduce to the special case where $Q_{\pi}$ itself is a non-degenerate quadric. 
Precisely, we assume the following.
\begin{item}
\item[($\dagger$)]  \label{I:cond1} $\pi$ an involution in \(S_{n}\). If \( \pi = (a_{1},b_{1}),\ldots,(a_{k},b_{k})(c_{1})\ldots(c_{m})\) in cycle representation then \( Q_{\pi} = \sum_{i=1}^{k}x_{a_{i}}x_{b_{i}} + \sum_{j=1}^{m}x^{2}_{c_{j}}.\)
\end{item}

In this special case it is possible to calculate the limit using the higher adjugate map, see \cref{Ss:compq}. We will
present an alternate calculation. We will use heavily use the notation introduced in \S \ref{SS:repthr}.

Let us first assume that $\pi = [1,2,\ldots,n]$ is the trivial permutation in $S_{n}$ so $Q_{\pi} = \sum x_{i}^{2}$. In this case the
complete quadric $Q_{\pi}$ corresponds to $h$ in the $G$-representation $\mbb{V}$, and $h$ corresponds to  vector $(1,
\ldots, 1)$ in $\mbb{A}^{n-1}$. Concretely, under this correspondence, we get 
\begin{equation}\label{E:general-case}
  \lambda(t) \cdot h  \longleftrightarrow \sum_{\alpha \in \Delta} t^{\langle \lambda, - 2\alpha \rangle}e_{\alpha}.
\end{equation}

Note that the fundamental weights and simple roots of $SL_{n}$ are related by  $\alpha_{i} = \epsilon_{i}
-\epsilon_{i+1}$. Substituting this in \cref{E:general-case} we note that for any admissible $1$-psg the limit as $t\rarr 0$
is the origin in $\mbb{A}^{n-1}$. The correspondence between the toric stratum in $\mbb{A}^{n-1}$ and the $G$-orbit closures
in $\mcal{X}_{n}$ shows that the origin corresponds to the minimal closed orbit (see \cref{E:toric-strata}). This proves the assertion when
$\pi$ is the trivial permutation.

When $\pi$ is an arbitrary involution, still satisfying condition ($\dagger$) above, we note that $\pi$ is of the
form $w \cdot [1,2,\ldots,n]$ for some involution $w \in S_{n}$. In this case $\pi$ belongs to $\overline{T \cdot
  h}$ (from highest weight consideration), so acting by a $1$-psg $\lambda = \lambda_{\mbf{a}}$ we get  
\begin{equation}\label{E:special-case}
     \lambda(t) \cdot(w \cdot h) = \sum_{\mu} t^{ \langle \lambda, w \cdot (\mu - \rho)\rangle }h_{\mu} \leftrightarrow
    \sum_{\alpha} t^{\langle  \lambda, -2 (w \cdot \alpha) \rangle} \cdot e_{\alpha} =   \sum_{i=1}^{n-1} t^{2(a_{w(i+1)}-a_{w(i)})} \cdot e_{\alpha_{i}} 
  \end{equation}
The last equality above follows from the Weyl group action -- $ w \cdot \epsilon_{i} = \epsilon_{w(i)}$.Consider the limit
\begin{equation}\label{E:limit-case}
  \lim_{t \rarr 0} \left ( \sum_{i=1}^{n-1} t^{2(a_{w(i+1)}-a_{w(i)})} \cdot e_{\alpha_{i}} \right ).  
\end{equation}

Whenever $w(i+1) > w(i)$, the coefficient of $e_{\alpha_{i}}$ is zero, and hence the limit point will belong to the
strata $S \subset \Delta$ which is the complement of  descent set of the permutation $w$. This shows that the barred
permuation corresponding to the limit point indeed matches the description of $\tau(\pi)$. This proves the assertion. 
\end{proof}

\begin{exmple}\label{E:sigma}

Let $\pi = (68) | (25)(4)(9) | (13)(7)$. Then the flag underlying $Q_{\pi}$ is a two step flag
$\mathcal{F} : 0 = V_0 \subset V_1 \subset V_2 \subset k^9$
where successive quotients are spanned respectively by the standard basis vectors $\{e_6, e_8\}$, $\{e_2, e_4, e_5, e_9\}$
and $\{e_1, e_3, e_7\}$. The non-degenerate quadrics on the successive quotients are given by $Q_1 = x_6 x_8$, $Q_2 = x_2 x_5 + x_4^2 + x_9^2$ and
$Q_3 = x_1 x_3 + x_7^2$.  The quadric $Q_1 = x_6 x_8$ is $T$-fixed so it is
also $\lambda$-fixed.  The quadric $Q_2 = x_2 x_5 + x_4^2 + x_9^2$ is not
$\lambda$-fixed and
\[
\lambda(t) \cdot Q_2 = t^{-(a_2 + a_5)} x_2 x_5 + t^{-2a_4} x_4^2 + t^{-2a_9} x_9^2.
\]
Since $\lambda$ is admissible, $2 a_9 > a_2 + a_5 > 2 a_4$ and it follows that
$\lim\limits_{t \rightarrow 0} \lambda(t) \cdot Q_2$ is the sequence of quadrics
$x_4^2, x_2 x_5, x_9^2$.
A similar calculation for $Q_3 = x_1x_3 + x_7^2$ yields
\[
\lim_{t \rightarrow 0} Q_{\pi} = Q_{\pi'}
\]
where $\pi' = (68) | (4) | (25) | (9) | (13) | (7)$ corresponding to the barred permutation
$[86 | 4 | 52 | 9 | 31 | 7]$.
\end{exmple}

\begin{rem}
  As the proof of Proposition \ref{prop:quadric-limit} illustrates, our adhoc definition admissible $1$-psg is not
  conceptually necessary.
  Any other choice, as long as it is sufficiently generic, will be
  related to our choice by the action of a Weyl group element. In our experience, the choice we have made leads to
 simplest results in terms of the indices.   
\end{rem}

\subsubsection{The map $\sigma$}

Next, we define a map in the opposite direction
\[
\sigma : \{ \text{barred permutations} \} \rightarrow \{ \text{degenerate involutions} \}
\]
which describes the maximal dimensional $B$-orbit contained in a BB-cell. 

\begin{defn}\label{D:asc-desc}
Let $\alpha = [\alpha_1 | \alpha_2 | \dots | \alpha_k]$ be a barred permutation. Let $d_j$ denote the largest value occurring in $\alpha_j$, giving rise to a sequence $\mathbf{d}_{\pi} = (d_1, d_2, \dots, d_k)$.  For example, if $\alpha = [ 86 | 9 | 52 | 4 | 7 | 31 ]$, then $\mathbf{d}_{\alpha} = (8,9,5,4,7,3)$.  We say that $\pi$ has a \emph{descent} (resp., \emph{ascent}) at position $i$ if $\mathbf{d}_{\pi}$ has a descent (resp., ascent) at position $i$.

\end{defn}

\begin{defn}\label{D:tau}
  The function $\sigma$ is constructed by the following recipe.
  \begin{itemize}

  \item Given a barred permutation $\pi$, construct $\sigma(\pi)$ by adding one-cycle $(i)$ for every length one string $i$
    appearing in $\pi$ and the two cycle $(ij)$ for every length two string $ji$ appearing in $\pi$.

    \item Remove bars from positions of ascent in $\pi$ and retain the bars at positions of descent at $\pi$. 
  \end{itemize}

\end{defn}

\begin{prop} \label{prop:dense-B-orbit}
  Given any barred permutation $\pi$, there is a unique $B$-orbit of maximum dimension that is contained in the BB-cell
  $X_{\pi}^{+}$. The $\mu$-involution indexing this $B$-orbit is given by $\sigma(\pi)$.
\end{prop}

\begin{proof}
  Note that $\tau(\sigma(\pi)) = \pi$, so it follows from \cref{prop:quadric-limit} that the
$B$-orbit indexed by $\sigma(\pi)$ is contained in the correct BB-cell. So we need to show that the $B$-orbit indexed by
the $\sigma(\pi)$ has the maximum dimension. It follows from the work of Brion and Luna, see \cite{BrionLuna87}, we know
that the intersection of a $G$-orbit and a BB-cell is either trivial or is an entire $B$-orbit, and moreover given any BB-cell there is a
maximum dimensional $G$-orbit such that the that the intersection is a dense $B$-orbit.

Note that Proposition \eqref{prop:quadric-limit} shows that the index corresponding to any $B$-orbit, contained in the BB-cell flowing to
$\pi$, must be obtained from $\pi$ by removing bars. The largest $B$-orbit will correspond to the degenrate
involution obtained by removing the maximum number of bars. 

Consider the sequence $\mbf{d}_{\pi}$ associated to $\pi = [\pi_{1}| \ldots |\pi_{k}]$, see \cref{D:asc-desc}. We will
show that we can always remove a bar at the position of ascent of $\mbf{d}_{\pi}$  where as removing a bar at a descent
is forbidden. Suppose a bar is removed at location $j$ of $\pi$, the resulting degenrate permuation $\pi^{\prime}$ is
shown below.

\begin{equation}\label{E:remove-bar}
    \begin{tikzcd}
        \pi = [\pi_{1}| \ldots | \pi_{j-1}| \pi_{j}| \ldots| \pi_{k}] \arrow[r, dashed] &  \pi^{\prime} \defeq
        [\pi_{1}|\ldots |\underbrace{\pi_{j-1}\; \pi_{j}}_{\text{Del. bar at} j}| \ldots| \pi_{k}] 
    \end{tikzcd}
\end{equation}

Suppose $\mbf{d}_{\pi}$ has ascent at $j$, then clearly $\tau(\pi^{\prime}) = \tau$. On the other hand, if
$\mbf{d}_{\pi}$ has a descent at $j$, then
\[\tau(\pi^{\prime}) = [\pi_{1}| \ldots| \pi_{j}| \pi_{j-1}| \ldots | \pi_{k} ]  \neq \pi. \]
So we are only allowed to remove bars at locations of ascent of $\mbf{d}_{\pi}$. This shows that $\sigma(\pi)$ is indeed
as claimed. 
\end{proof}

We consider an example to illustrate the proof of \cref{prop:dense-B-orbit}.
\begin{exmple} \label{E:flow}
  Consider the barred permutation $\pi = [86| 4| 52| 9| 31| 7]$ then as predicted by \cref{D:tau} we have $\sigma(\pi) =
  (68) | (25)(4)(9) | (13)(7)$. Let us consider the degenerate involution $\pi^{\prime} = [(86)(4)|52|9|31|7]$. We wish
  to show that the distinguished quadric $Q_{\pi^{\prime}}$ will not flow to $Q_{\pi}$ under an admissible $1$
  -psg. This is clear because the only non-degenerate quadric, which is not $T$-fixed, appearing in $Q_{\pi^{\prime}}$
  is $x_{8}x_{6}+  x^{2}_{4}$.  Under the flow of an admissible $1$-psg, see \cref{E:sigma}, it will flow to the point $[4|86|52|9|3|7] \neq \pi$.  
\end{exmple}

Given a barred permutation $\pi$, let $w(\pi)$ denote the unique permutation in $S_{n}$ which in one-line notation is
obtained by removing all bars.

We let  $\text{inv}(\pi)$ denote the number of length two strings that occur in $\pi$ and let $\text{asc}(\pi)$
denote the number of ascents in $\mbf{d}_{\pi}$. Let $w_{0}$ denote the element of maximal length in $S_{n}$.

\begin{lema} \label{L:dim-estimate}
The dimension of the $B$-orbit $\tau(\pi)$, corresponding to a barred permutation $\pi$, is given by $\ell(w_0) -
\ell(w(\pi)) + \text{inv}(\pi) + \text{asc}(\pi)$, where $\ell(-)$ is the length function on $S_{n}$. 
\end{lema}

\begin{proof}
Since $w(\pi)$ belongs to the $\WW$-set of the $B$-orbit containing $Q_\pi$, the codimension of the $B$-orbit containing
$Q_{\pi}$ in its $G$-orbit is $\ell(w(\pi))$; see \cite{CanJoyceWyser2}. The codimension of the closed $B$-orbit in the
$G$-orbit containing  $Q_{\pi}$ is $\text{inv}(\pi)$ and the dimension of the closed $G$-orbit is $\ell(w_0)$, it
follows that the dimension of the $B$-orbit containing $Q_{\pi}$ is $\ell(w_0) + \text{inv}(\alpha) -
\ell(w(\alpha))$. The result follows from the fact that the codimension of the $B$-orbit containing $Q_{\pi}$ in its closure is $\text{asc}(\pi)$.
\end{proof}

\subsubsection{Concluding Remarks}
In  Figure~\ref{fig:Cells}, we depict the cell decomposition of $\mathcal{X}_3$, each colored rectangle represents a
$B$-orbit parametrized by its corresponding $\mu$-involution, and the edges stand for the covering relations in Bruhat
order.  A BB-cell is a union of all Borel orbits of the same color.

We illustrate the resulting cell decomposition, when $n = 3$, in Figure \ref{fig:cell-decomp3}. The dimension of a cell
 corresponding to a vertex in the figure is equal to the length of any chain from the bottom cell.
 A vertex corresponding to a cell $X_{\pi}^{+}$ is connected by an edge to a vertex of a cell $X_{\pi^{\prime}}^{+}$ of dimension
 one lower if and only if $X_{\pi^{\prime}}^{+} \subseteq \overline{X_{\pi}^{+}}$.

\begin{figure}[htp]
\begin{minipage}{0.75\linewidth}
\centering
\begin{tikzpicture}
\begin{scope}[rotate = 45, scale = 0.35, every node/.style = {scale=0.65}]

\node [shape=rectangle,draw,fill=white] at (0,0) (a) {$[3|2|1]$};

\node [shape=rectangle,draw,fill=purple] at (-9,5) (b1) {$[3|21]$};
\node [shape=rectangle,draw,fill=blue!30!] at (-3,5) (b2) {$[32|1]$};
\node [shape=rectangle,draw,fill=red!20!] at (3,5) (b3) {$[2|3|1]$};
\node [shape=rectangle,draw,fill=orange!90!] at (9,5) (b4) {$[3|1|2]$};

\node [shape=rectangle,draw,fill=green!50!] at (-12.25,10) (c1) {$[2|31]$};
\node [shape=rectangle,draw,fill=orange!90!]  at (-7.25,10) (c2) {$[3|12]$};
\node [shape=rectangle,draw,fill=red!20!] at (-2.25,10) (c3) {$[23|1]$};
\node [shape=rectangle,draw,fill=brown!50!] at (2.25,10) (c4) {$[31|2]$};
\node [shape=rectangle,draw,fill=red!50!] at (7.25,10) (c5) {$[1|3|2]$};
\node [shape=rectangle,draw,fill=green!15] at (12.25,10) (c6) {$[2|1|3]$};

\node [shape=rectangle,draw,fill=green!50!] at (-12.25,15) (d1) {$[321]$};
\node  [shape=rectangle,draw,fill=brown]  at (-7.25,15) (d2) {$[1|32]$};
\node [shape=rectangle,draw,fill=green!15] at (-2.25,15) (d3) {$[2|13]$};
\node  [shape=rectangle,draw,fill=red!50!]  at (2.25,15) (d4) {$[13|2]$};
\node [shape=rectangle,draw,fill=yellow!10!] at (7.25,15) (d5) {$[21|3]$};
\node  [shape=rectangle,draw,fill=yellow]  at (12.25,15) (d6) {$[1|2|3]$};

\node  [shape=rectangle,draw,fill=brown] at (-9,20) (e1) {$[132]$};
\node  [shape=rectangle,draw,fill=yellow!10!] at (-3,20) (e2) {$[213]$};
\node  [shape=rectangle,draw,fill=yellow]  at (3,20) (e3) {$[1|23]$};
\node  [shape=rectangle,draw,fill=yellow]  at (9,20) (e4) {$[12|3]$};

\node [shape=rectangle,draw,fill=yellow] at (0,25) (f) {$[123]$};

\filldraw[dashed, thick,fill=red] (a) to (b1);
\draw[dashed] (a) to (b2);
\draw[dashed] (a) to (b3);
\draw[dashed] (a) to (b4);

\draw[dashed] (b1) to (c1);
\draw[dashed] (b1) to (c2);
\draw[dashed] (b2) to (c3);
\draw[dashed] (b2) to (c4);
\draw[dashed] (b3) to (c1);
\draw[very thick] (b3) to (c3);
\draw[dashed] (b3) to (c5);
\draw[dashed] (b3) to (c6);
\draw[very thick] (b4) to (c2);
\draw[dashed] (b4) to (c4);
\draw[dashed] (b4) to (c5);
\draw[dashed] (b4) to (c6);

\draw[very thick] (c1) to (d1);
\draw[dashed] (c1) to (d2);
\draw[dashed] (c1) to (d3);
\draw[dashed] (c2) to (d2);
\draw[dashed] (c2) to (d1);

\draw[dashed] (c2) to (d3);
\draw[dashed] (c3) to (d1);
\draw[dashed] (c3) to (d4);
\draw[dashed] (c3) to (d5);
\draw[dashed] (c4) to (d1);
\draw[dashed] (c4) to (d4);
\draw[dashed] (c4) to (d5);

\draw[dashed] (c5) to (d2);
\draw[very thick] (c5) to (d4);
\draw[dashed] (c5) to (d6);

\draw[very thick] (c6) to (d3);
\draw[dashed] (c6) to (d5);
\draw[dashed] (c6) to (d6);

\draw[dashed] (d1) to (e1);
\draw[dashed] (d1) to (e2);

\draw[very thick] (d2) to (e1);
\draw[dashed] (d2) to (e3);

\draw[dashed] (d3) to (e2);
\draw[dashed] (d3) to (e3);

\draw[dashed] (d4) to (e1);
\draw[dashed] (d4) to (e4);

\draw[very thick] (d5) to (e2);
\draw[dashed] (d5) to (e4);
\draw[very thick] (d6) to (e3);
\draw[very thick] (d6) to (e4);

\draw[dashed] (e1) to (f);
\draw[dashed] (e2) to (f);
\draw[very thick] (e3) to (f);
\draw[very thick] (e4) to (f);
\end{scope}

\end{tikzpicture}
\end{minipage}
\caption{Cell decomposition and the Bruhat order for $\mathcal{X}_3$.}
\label{fig:Cells}
\end{figure}
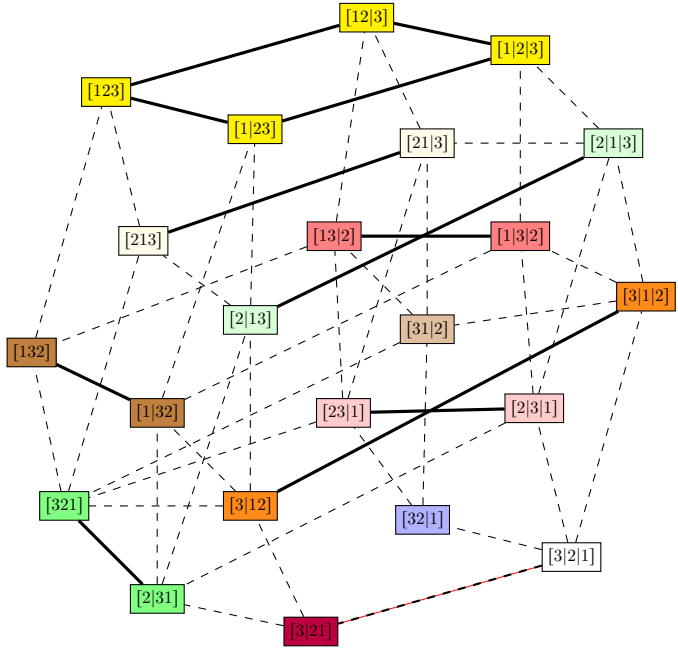

\begin{figure}[htp]
  
  \begin{minipage}{0.45\linewidth}
    \centering
    \begin{tikzpicture} 
  \begin{scope}[thick,scale=0.25, every node/.style={scale=0.6}]
\node [shape=rectangle,draw,fill=white] at (0,0) (a) {$[3|2|1]$};
\node [shape=rectangle,draw,fill=purple] at (-5,5) (b1) {$[3|21]$};
\node [shape=rectangle,draw,fill=blue!30!] at (5,5) (b2) {$[32|1]$};
\node [shape=rectangle,draw,fill=orange!90!] at (-10,10) (c1) {$[3|1|2]$};
\node [shape=rectangle,draw,fill=brown!50!] at (0,10) (c2) {$[31|2]$};
\node [shape=rectangle,draw,fill=red!20!] at (10,10) (c3) {$[2|3|1]$};
\node [shape=rectangle,draw,fill=green!15] at (-10,15) (d1) {$[2|1|3]$};
\node [shape=rectangle,draw,fill=green!50!] at (0,15) (d2) {$[2|31]$};
\node [shape=rectangle,draw,fill=red!50!] at (10,15) (d3) {$[1|3|2]$};
\node [shape=rectangle,draw,fill=yellow!10!] at (-5,20) (e1) {$[21|3]$};
\node  [shape=rectangle,draw,fill=brown]  at (5,20) (e2) {$[1|32]$};
\node  [shape=rectangle,draw,fill=yellow]  at (0,25) (f) {$[1|2|3]$};

\draw (a) to (b1);
\draw (a) to (b2);
\draw (b1) to (c1);
\draw (b2) to (c2);
\draw (b2) to (c3);

\draw (c1) to (d1);
\draw (c1) to (d2);
\draw (c2) to (d2);
\draw (c2) to (d3);
\draw (c3) to (d2);
\draw (c3) to (d3);

\draw (d1) to (e1);
\draw (d2) to (e1);
\draw (d2) to (e2);
\draw (d3) to (e2);

\draw (e1) to (f);
\draw (e2) to (f);
\end{scope}
\end{tikzpicture}
\caption{Poset of BB-cells of  $\mathcal{X}_3$}
\label{fig:cell-decomp3}
\end{minipage}
\hspace{1mm}
\begin{minipage}{0.45\linewidth}
    \begin{tikzpicture}
\begin{scope}[thick, scale=0.30, every node/.style={scale=0.55}]
\node [shape=rectangle,draw,fill=blue!20!] at (0,0) (a) {$[4|321]$};
\node [shape=rectangle,draw,fill=blue!20!]  at (-5,5) (b1) {$[3|421]$};
\node [shape=rectangle,draw,fill=blue!20!]  at (0,5) (b2) {$[4|132]$};
\node [shape=rectangle,draw,fill=blue!20!]  at (5,5) (b3) {$[4|213]$};
\node [shape=rectangle,draw,fill=white]  at (-7.5,10) (c1) {$[2|431]$};
\node [shape=rectangle,draw,fill=blue!20!]  at (-2.5,10) (c2) {$[3|142]$};
\node [shape=rectangle,draw,fill=blue!20!] at (2.5,10) (c3) {$[3|214]$};
\node [shape=rectangle,draw,fill=blue!20!]  at (7.5,10) (c4) {$[4|123]$};
\node [shape=rectangle,draw,fill=white]  at (-7.5,15) (d1) {$[1|432]$};
\node [shape=rectangle,draw,fill=blue!20!]  at (-2.5,15) (d2) {$[2|143]$};
\node [shape=rectangle,draw,fill=white]  at (2.5,15) (d3) {$[2|314]$};
\node [shape=rectangle,draw,fill=blue!20!]  at (7.5,15) (d4) {$[3|124]$};
\node [shape=rectangle,draw,fill=white]  at (-5,20) (e1) {$[1|243]$};
\node [shape=rectangle,draw,fill=white]  at (0,20) (e2) {$[1|324]$};
\node [shape=rectangle,draw,fill=blue!20!]  at (5,20) (e3) {$[2|134]$};
\node [shape=rectangle,draw,fill=white]  at (0,25) (f) {$[1|234]$};

\draw[-,very thick, violet] (a) to (b1);
\draw[-,very thick, violet] (a) to (b2);
\draw[-,very thick, violet] (a) to (b3);
\draw[-,thick] (b1) to (c1);
\draw[-,very thick, violet] (b1) to (c2);
\draw[-,very thick,violet] (b1) to (c3);
\draw[-,very thick,violet] (b2) to (c2);
\draw[-,very thick,violet] (b2) to (c4);
\draw[-,thick] (b3) to (c1);
\draw[-,very thick,violet] (b3) to (c3);
\draw[-,very thick,violet] (b3) to (c4);
\draw[-,thick] (c1) to (d1);
\draw[-,thick] (c1) to (d2);
\draw[-,thick] (c1) to (d3);
\draw[-,thick] (c2) to (d1);
\draw[-,very thick,violet] (c2) to (d2);
\draw[-,very thick,violet] (c2) to (d4);
\draw[-,thick] (c3) to (d3);
\draw[-,very thick,violet] (c3) to (d4);
\draw[-,thick] (c4) to (d1);
\draw[-,very thick,violet] (c4) to (d2);
\draw[-,very thick,violet] (c4) to (d4);
\draw[-,thick] (d1) to (e1);
\draw[-,thick] (d1) to (e2);
\draw[-,thick] (d2) to (e1);
\draw[-,very thick,violet] (d2) to (e3);
\draw[-,thick] (d3) to (e2);
\draw[-,thick] (d3) to (e3);
\draw[-,thick] (d4) to (e2);
\draw[-,very thick,violet] (d4) to (e3);
\draw[-,thick] (e1) to (f);
\draw[-,thick] (e2) to (f);
\draw[-,thick] (e3) to (f);
\end{scope}
\end{tikzpicture}
\end{minipage}
\caption{The dense $B$-orbits of BB-cells in $\mcal{X}_{3}$ wrt closure order.}
\label{fig:in-1+3}
\end{figure}

For $n \geq 3$, the Bia{\l}ynicki-Birula decomposition of $\mathcal{X}_n$ is not a stratification. To see this, we
consider the Bruhat order on $\mathcal{X}_3$, depicted in Figure~\ref{fig:Cells}. The closure of the pink cell
$X^{+}_{[1|3|2]}$ intersects the orange cell $X^{+}_{[3|1|2]}$ in the $B$-orbit $\mscr{O}^{[3|1|2]}$, which is non-empty, but not equal to the entire orange cell which also includes $\mscr{O}^{[3|12]}$.

It is desirable to have a combinatorial rule determining the (covering) relations of Bruhat order which does not go
through the costly inductive procedure given in Section \ref{S:Bruhat}. Given a composition $\mu$ of $n$, let us denote
by $\mc{B}_{Cell}(\mu)$ the set of all $\mu$-involutions $\pi_{\mu}$ such that Borel orbit $\mscr{O}^{\pi_{\mu}}$ is
dense in its corresponding BB-cell.

Experimentally, we have observed that the inclusion order restricted to $\mc{B}_{Cell}(\mu)$ is a ranked poset with a
minimal and a maximal element. For example this is depicted in Figure
\ref{fig:in-1+3}  where we consider the $\mcal{B}_{Cell}(1,3)$
as an embedded sub-poset in the closure order on all $(1,3)$-involutions. However we are unable to establish it in
general and we pose it as a conjecture. 

\begin{conj}
Fix a BB-cell decomposition of $\mcal{X}_{n}$ and consider all the Borel-orbits in $\mcal{X}_{n}$  which are dense in
some BB-cell. Denote this set by $\mc{B}_{Cell}(\mathcal{X})$. Then the Bruhat order on Borel orbit restricted to $\mc{B}_{Cell}(X)$ is a graded poset with a maximum and a minimum element.
\end{conj}

\printbibliography

\end{document}